\documentclass{amsart}

\usepackage{graphicx}

\usepackage{amsmath,amssymb}
\usepackage{color}

\newtheorem{theorem}{Theorem}[section]

\newtheorem{corollary}[theorem]{Corollary}
\newtheorem{example}[theorem]{Example}

\newtheorem{definition}[theorem]{Definition}
\newtheorem{remark}[theorem]{Remark}

\newcommand\NCP{\mathcal{NC}}
\newcommand\ee{\varepsilon}

\newcommand\EE{\mathbb{E}}
\newcommand\FF{\mathbb{F}}
\newcommand\CC{\mathbb{C}}
\newcommand\RR{\mathbb{R}}
\newcommand\cA{\mathcal{A}}

\newcommand\cB{\mathcal{B}}

\newcommand\SP{\mathcal{P}}
\newcommand\IP{\mathcal{I}}

\begin{document}

\title{Non-Commutative Probability Theory for Topological Data Analysis}

\begin{abstract}
Recent developments  \cite{Par15} have found unexpected connections between non-commutative probability theory and algebraic topology. In particular, Boolean cumulants functionals seem to be important for describing structure morphisms of homotopy operadic algebras \cite{DPT15a}.

We provide new elementary examples which show a connection between non-commutative probability theory and algebraic topology, based on spectral graph theory. These observations are important for bringing new ideas from non-commutative probability into TDA and stochastic topology, and in the opposite direction.
\end{abstract}
\author{Carlos Vargas (CONACyT-CIMAT)}

\maketitle

\section{Introduction}

One of the main concepts in non-commutative probability theory is free independence. Free independence and free probability theory were initially motivated by problems on the classification theory of von-Neumann algebras \cite{Voi85,Voi96, VDN92} but have evolved into a rich field with connections to several areas in mathematics, such as representation theory \cite{Bia98, Col03, CS06}, combinatorics \cite{Spe94, NS06, Spe98} and spectral graph theory \cite{ALS07, Oba08}.

The first of these important links was found by Voiculescu in \cite{Voi91}, where he showed that free independence is not just an abstract notion in the framework of infinite dimensional operator algebras, but a concretely realizable relation, leading to groundbreaking multivariate generalizations of the pioneering theorems of Wigner \cite{Wig58} and Marchenko-Pastur\footnote{To be strict here, we mean multivariate generalizations for the weak versions of these theorems.} \cite{MP67} on asymptotic random matrix theory. 

Free independence describes exactly the asymptotic global collective behavior of the different matrices involved in the models of Wigner and Marchenko-Pastur. The more sophisticated notion of an operator-valued probability space $(\cA,\cB,\FF)$ and the corresponding  $\cB$-free independence \cite{Voi95, Shl96} allowed to treat more recent matrix models and combinations of these. The developments of the theory (in particular, the works on analytic subordination \cite{Bia98a, BB07} and algebraic linearization \cite{HT05,And13}) have led to a powerful toolbox for computing asymptotic global eigenvalue distributions \cite{BSTV14, BMS15, ANV16, BSS15}.

A combinatorial approach to non-commutative independence started with the works of Speicher on the cumulant characterizations of free independence \cite{Spe94} (see also \cite{NS06}). The combinatorial picture of free probability offered a new perspective which led to the definition of new notions of independence \cite{SW97, Mur01} and classification results on these \cite{Spe97, Mur03}. In particular, the approach in terms of co-products on the different categories of algebras, followed by \cite{BGS02}, explains the deep parallelism between the different non-commutative probability theories. Some years later, simple realizations of independent random variables and their convolutions were found by the means of spectral graph theory \cite{ABO04, Oba08, ALS07}.

Recent works \cite{DPT15a, DPT15b, Par15} have found unexpected connections between cumulant functionals in non-commutative probability and morphisms of homotopy operadic algebras.

The aim of this work is to present an alternative way to consider ``topological" non-commutative random variables. We show that topological spaces may be considered as random variables in such a way that their non-commutative distribution encodes its topological properties. 

For this, it will be important to consider the operator-valued probability space where $\cA=M_N(\CC)$ is the algebra of complex deterministic matrices $\cB=\langle P_0,\dots,P_d \rangle\subset \cA$ is a $(d+1)$-dimensional algebra, generated by self-adjoint, pairwise orthogonal (and hence commutative) projections and $\FF:\cA\to \cB$ is the usual conditional expectation compatible with the normalized trace $\frac{1}{N}\mathrm{Tr}$. For these situations, called \emph{rectangular (NC)-probability spaces}, most tools, both combinatorial and analytical, can be lifted directly from the scalar-valued level to the $\cB$-valued level. 

Despite of their simplicity, rectangular spaces have already been quite useful for sorting out distributions in random matrix theory. 

For example, the random matrix models for wireless communications usually involve rectangular matrices of different sizes.  However, Voiculescu's original results on asymptotic free independence do not hold unless the matrices involved in the models are square and growing at the same rates. As observed by Benaych-Georges in \cite{BG09, BG09a}, this problem can be fixed by considering free independence with respect to a suitable rectangular space, where the traces of the projections are proportional to the different sizes of the matrices in consideration (see, for example Chapter 6 of \cite{CD11}, for a series of examples, and \cite{SV12,Var16}, for their treatment using free probability). 

As another example, in \cite{ANV16}, the authors study block-modifications of unitarily invariant random matrix ensembles. Considering rectangular spaces allowed uniform treatments for several instances of this problem \cite{BN12a, FS13, JLN14, JLN15}. 

For this paper, the rectangular distributions will help us in two situations: 
\begin{itemize}
\item They will allow us to connect non-commutative probability and algebraic topology, by granting us direct access to the chain groups of different dimensions for a simplicial complex $X$ (and, as a consequence, to its Betti numbers).
\item They allow us to construct basic objects in statistics (such as histograms). 
\end{itemize}
The combination of these two observations gives some more heuristics on why the Betti numbers may be used in statistics (as in TDA), which we will further elaborate in several directions. In particular, we recall how cycles can be filtered according to local topological properties using Boolean cumulants (Section \ref{subsection:graphprod}), we show that boundary matrices are natural candidates for non-commutative random variables, with distributions encoding Betti numbers (see Section \ref{section:NCDSC}), we discuss similarities between crucial objects in TDA and operator-valued free probability (Section \ref{subsection:OVFP}), and propose new models with superconvergence in TDA (see Section \ref{subs:repulsionTDA}).

For our purposes, we rely heavily on the transparent realizations of classical, Boolean, monotone and free convolutions of random variables as graphs products (see \cite{ABO04, HO07, ALS07,Oba08}). Graphs can be thought as the simplest examples of both topological spaces and non-commutative random variables.
The usual notions of non-commutative distributions for graphs are weighted versions of the (discrete) spectral probability distributions of their adjacency matrices.

We show that, for a simplicial complex $X=\bigcup_{j\leq d} \bigcup_{i\leq n_j}\sigma^{(j)}_i$, a canonical rectangular probability space may be considered $(\cA,\cB):=(M_N(\CC),\langle P_0,\dots, P_d\rangle)$ and the $\cB$-distributions of the incidence $I(X)\in \cA$ and boundary $J(X)\in \cA$ matrices are natural candidates for its algebraic distribution. This because, in particular, for the case $d=1$, where $X$ is a graph, $I(X)^2=J(X)^2=0$ and hence their $\cB$-valued distributions are accessible and almost include the usual distribution of $G$ in terms of the adjacency matrix.

For general $d\geq 2$, $I(X)^2$ is no longer zero. However, the fundamental topological-algebraic equations $\partial_{i-1}\partial_{i}=0$ of the boundary operators translate into the desired equation $J(X)^2=0$. The $\cB$-distribution of $J(X)$ is easier to compute (it depends only on the positive orthogonal operators $JJ^*$ and $J^*J$) and organizes topological information about $X$ in a better way.

In particular the element $\mathcal L:=JJ^*+J^*J$ is a diagonal, positive non-commutative random variable, with analytic distribution completely determined by a $(d+1)$-tuple of probability measures on the positive real line $\mu=(\mu_0,\mu_1,\dots,\mu_d)$. It can be shown that the weight of the atom at zero of $\mu_i$ is $n_i^{-1}\beta_i(X)$, where $\beta_i(X)$ its i-th Betti number of $X$. 

These results are not new. They follow from Eckmann's \cite{Eck44} study of the spectra of the so-called higher-order combinatorial Laplacians $\mathcal L_i=P_i \mathcal L P_i$. In \cite{Fri98}, Friedman designed powerful algorithms to compute such spectra, based on Hodge Theory. In \cite{HJ13}, the authors study the effects on these spectra, of weight functions and operations of simplicial complexes (e.g. joins, wedge sums, etc). 

The novelty of our work is to interpret $J$ (or its dual) as a concrete example of a ``topological" non-commutative random variable, yielding a direct connection between non-commutative distributions and the seemingly unrelated, Betti numbers. 

This paper should then serve as a motivation to follow the advise in \cite{Par15} of developing operator-valued, homotopy probability theory. 
Here we point out that, in some sense, one must consider operator-valued distributions to provide basic examples of topological random variables.  Most scalar-valued topological invariants, such as the Euler characteristic, are typically calculated in practice in terms of the number of faces $(n_i)_{i\geq d}$ or the Betti numbers $(\beta_i)_{i\geq d}$, so the rectangular structure of the space seems hard to escape.

From the point of view of probability and statistics, there has been an increasing interest during the last decades on the statistical use of Betti numbers  and algorithms for their effective computation, in part, due to the surge of topological data analysis (see \cite{ELZ02, CZCG04}, and the surveys \cite{Car09}, \cite{Kah14} on TDA and random simplicial complexes).

The main idea of TDA is to provide new ways to distinguish statistically from different data sets (or point clouds) in metric spaces. In practice these data sets come from measurements of complicated objects of quite different natures (e.g. 3D-coordinates of forests of brain arteries or precolonial Mexican masks, protein arrangements, collections of chunks of DNA, etc.). For each point cloud, one should construct a process $\mathbb X_i:=(X_t^{(i)})_t$ of topological spaces  $a\leq t\leq b$. For example, \v{C}ech complexes or Vietoris-Ripps complexes are popular choices. 

The idea of considering such complexes is that they will allow us to encode the \emph{shape of the data} through their \emph{topological invariants at different scales}. For example, one possibility is to consider the barcodes of Betti numbers (or Betti curves):
\begin{eqnarray*}
(\dim\circ H(X_t^{(i)}))_{a\leq t\leq b}&=&(\dim\circ (H_0(X_t^{(i)}),H_1(X_t^{(i)}), \dots ,H_d(X_t^{(i)})))_{a\leq t\leq b}\\
&=&(\beta_0(X_t^{(i)}),\beta_1(X_t^{(i)}), \dots ,\beta_d(X_t^{(i)}))_{a\leq t\leq b},
\end{eqnarray*}
where $H(X):=(H_i(X))_{i\leq d}$ is the homology functor. The complexes considered in TDA are usually very large and hence effective methods in computational algebraic topology need to be invoked. For instance, for the case of the Betti numbers, one possibility is to sacrifice some additional topological information, by computing homology with coefficients in the field of two elements $\FF_2$, in order to gain access to algorithmic methods for sparse matrices to compute them.

Although the Betti curves clearly capture important topological features of the space, one is ultimately interested in a topological summary $T(\mathbb X_i)$, that should not only preserve these features, but \emph{use them} to cluster the data sets in a way that we may distinguish between different kinds (e.g. DNA chunks from sick vs control individuals, gender of the individual with a given forests of brain arteries, origin of a precolonial mask, etc.). In other words, $T(\mathbb X_i)$ should work as a \emph{statistic}. 

The main technical difficulty here is that the space where $T(\mathbb X_i)$ takes values is not exactly a textbook example in statistics, and thus, extending the statistical toolbox to such complicated situations is an active area of research and a core problem in TDA.

The most important examples of statistics in TDA are the persistence diagrams, which favor those topological features that are present during longer time intervals, whereas the brief topological features are weighted out or ignored.

Recently, Patel \cite{Pat16} showed that quite general persistence diagrams are obtained by performing M\"obius inversions of homology-group valued functions with respect to the poset of persistence diagrams $(\mathrm{Dgm},\subseteq)$, consisting of multi-sets of half-open intervals $[p,q)$ with $p<q$ and half-infinite intervals $[p,\infty)$, with the usual partial order of contention. Patel's result is framed in Bubenik's categorical approach to persistence diagrams \cite{BS14}.  

We should at this point remind the reader that a M\"obius inversion (w.r.t. the subsets $\IP, \NCP$ of set partitions $\SP$) is at the backbone of the different non-commutative probability theories. We intentionally split the Betti curves as $\beta=\dim \circ H$ to highlight a remarkable similarity with situations often encountered when dealing with random matrix models by the means of free probability. 

The topological space $X_t$ is thought in practice as a large random sparse matrix, and $H$ is an array of projections. The analogy is given by replacing $\dim$ with $\frac{1}{N}\mathrm{Tr}$ and $H$ with some conditional expectation, depending on the matrix model, where the actual computations and the M\"obius inversions (cumulants) take place. After applying $\dim$ (resp. $\frac{1}{N}\mathrm{Tr}$), one obtains useful statistical functions, such as persistence landscapes \cite{Bub15} (resp. the desired scalar-valued Cauchy-Stieltjes transform).

In addition, the poset of set partitions models compatible conditional expectations and the interval partitions (which govern Boolean probability) are naturally relevant in order to construct useful histograms (which can be defined in terms of conditions expectations to rectangular spaces). Disconnected bins, which may be seen as corresponding to non-interval partitions, provide poor geometrical information about the outcome of an experiment. The combinatorial similarities between Betti numbers and Boolean cumulants suggest that a higher dimensional notion of Boolean cumulants may be defined (\cite{GMV17}). 

Thus, Betti numbers can be obtained directly from the non-commutative distribution of a simplicial complex, viewed as a non-commutative random variable, and, by Patel's result, persistence diagrams are obtained from them by M\"obius inversions. We conclude that persistence diagrams should be thought as cumulants. Furthermore, the Betti numbers are just some small piece of information contained in the non-commutative distribution of $J$ (i.e. the weights at zero of the analytic distribution of $\mathcal L$). A more delicate analysis of the spectra of $\mathcal L_i$, (as performed in \cite{HJ13}) could probably be used for constructing more general statistics in TDA.

We assume that the reader is familiar with basic concepts in algebraic topology. Apart from the introduction, the paper is organized as follows: 

In Section 2 we introduce the main definitions and main examples from non-commutative probability. We define the different notions of independence and their characterizations on terms of cumulants. In Section 3 we present realizations of independent random variables in terms of graphs, their products, adjacency matrices and spectral measures. 

This motivates our main definitions in Section 4 of possible algebraic distributions for simplicial complexes. We emphasize on the distribution given by the boundary matrix $J(X)$. We show that important topological invariants are extracted from a natural operator-valued non-commutative distribution of $J(X)$ and that, for the case of graphs, the new notion of distribution almost coincides (or, more precisely, almost includes) the previous one, defined in terms of the adjacency matrix. 

Section 5 contains a very brief discussion introduction to TDA and its potential connections with non-commutative probability theory. 

Section 6 gives some basic definitions and results on large random matrix theory and a quick insight on the algorithmic aspects of operator-valued free probability. The random matrix models are also used to produce toy models with repulsion in Section 7, where we point out directions for future work. 

\subsection*{Acknowledgments}
I am grateful to M. Nakamura, V. Perez-Abreu, R. Biscay and specially to F. Reveles, for their 2016 course ``Probability and Statistical Inference in TDA", and to the organizers and lecturers of the 2nd and 3rd Schools on TDA (UNAM Juriquilla, Quer\'etaro 2015 and ABACUS Centre, Estado de M\'exico 2017), in particular, to P. Bubenik, who taught courses in both events and informed me about the reference \cite{Pat16}). All these courses allowed me to gain some insight into algebraic topology and TDA. 

I also appreciate encouraging discussions with O. Arizmendi, T. Gaxiola, J. Cervantes, R. Speicher, R. Friedrich, A. Rieser and thank T. Hasebe and O. Arizmendi for respectively informing me about homotopy probability theory and higher-order combinatorial Laplacians, which led me to the references \cite{DPT15a, DPT15b, Par15} and \cite{Eck44,HJ13}. I deeply thank G. Flores and Y. Hernandez for producing some of the figures included in this paper.

\section{Algebraic probability spaces, non-commutative independence and cumulants}

The notion of a non-commutative probability space (NCPS) allows to consider classical random variables and deterministic matrices in the same framework. The key idea is to abstract algebraically the fundamental notions of probability theory, in terms of the expectation (and, more generally, conditional expectations). 

Voiculescu came up with the notion of free independence, which allowed him to import several ideas from classical probability into a parallel, non-commutative framework, sitting inside the spectral theory of operator algebras. Free independence was later shown to be crucial for a more conceptual understanding of the global behavior of eigenvalues of large random matrices. 

From the combinatorial point of view, a notion of independence can be seen as a rule for factorizing mixed moments. The different notions of non-commutative independence are classified very neatly in terms of cumulant functionals with respects to different posets of partitions. 

\subsection{Main definitions and basic examples}

\begin{definition}
(1). A NCPS is a pair $(\cA,\tau)$, where $\cA$ is a unital $*$-algebra and $\tau:\cA\to\CC$ is a unital positive linear functional. The functional $\tau$ should be understood as playing the role of the expectation. In this spirit, we call $a_1,a_2,\dots, a_m\in \mathcal A$ \emph{non-commutative random variables} and any expression of the form $\tau(a_{i(1)}^{\varepsilon_1}a_{i(2)}^{\varepsilon_2}\dots a_{i(k)}^{\varepsilon_k})$, for $k\geq 1$, $1\leq i(1),\dots,i(k)\leq m$, and  $\varepsilon=(\varepsilon_1,\dots,\varepsilon_k)\in \{1,*\}^k$ is called a \emph{mixed moment} of $(a_1,\dots, a_k)$.

Two tuples of random variables $(a_1,\dots,a_k)\in\cA_1^k, (b_1,\dots,b_k)\in\cA_2^k$ in (possibly different) non-commutative probability spaces $(\cA_i,\tau_i)$ have the same \emph{(algebraic) distribution} if all the $*$-moments coincide, that is $$\tau_1(a_{i(1)}^{\varepsilon_1}a_{i(2)}^{\varepsilon_2}\dots a_{i(k)}^{\varepsilon_k})=\tau_2(b_{i(1)}^{\varepsilon_1}b_{i(2)}^{\varepsilon_2}\dots b_{i(k)}^{\varepsilon_k}),$$ for all $k\geq 1$, $1\leq i(1),\dots,i(k)\leq m$, and  $\varepsilon\in \{1,*\}^k$.

A non-commutative random variable $a\in \cA$ may be of special type: 
\begin{itemize}
\item Positive: $a=bb^*$ for some $b\in \cA$
\item Self-adjoint: $a=a^*$
\item Unitary: $a^*=a^{-1}$
\item Normal: $aa^*=a^*a$.
\end{itemize} 

For these kind of variables, the algebraic non-commutative distribution of a \emph{single} random variable $a\in \cA$ may be simplified and encoded into a probability measure, supported, respectively, on the positive real line, on the real line, on the unit circle and on the complex plane.

(2). More generally, an operator-valued probability space (OVPS, or $\cB$-probability space) is a triple $(\cA,\cB,\FF)$, where $1_{\cA}\in\cB\subseteq\cA$ is a sub-algebra of $\cA$ containing the unit and $\FF:\cA\to \cB$ is a conditional expectation, that is, a $\cB$-linear map satisfying, for all $a\in \cA, b,b'\in \cB$:
$$\FF(bab')=b\FF(a)b', \quad \FF(1_{\cA})=1_{\cA}.$$  Observe that this implies that $\FF(b)=b$ for all $b\in \cB$. We say that two operator-valued structures $(\cA,\cB_1, \FF_1), (\cA,\cB_2, \FF_2)$ on the same space are compatible if $\cB_1\subseteq \cB_2$ and $\FF_1\circ \FF_2=\FF_1$.
\end{definition}

The main examples are algebras of random variables and matrices.

\begin{example}\label{ex:ovps}
(1). In the framework of a classical probability space $(\Omega, \mathcal F, \mathbb P)$, we may consider $\cA:=\cA_{\mathcal F}$ to be the algebra of $\mathcal F$-measurable complex-valued random variables with compact support and the usual expectation $$\tau:=\EE:X\mapsto \int_{\Omega}X(\omega)\mathrm{d}\mathbb P(\omega),$$ and adjoint given by the complex conjugate $X^*:=\bar X$.

Furthermore, any sub-$\sigma$-algebra of $\mathcal H \subseteq\mathcal F$ induces a conditional expectation  $\EE_{\mathcal H}:\cA_{\mathcal F}\to \cA_{\mathcal H}$. Indeed for any $X\in \cA_{\mathcal F}$, it is well-known that there exist a unique $\mathcal H$-measurable random variable $\EE_{\mathcal H}(X)$ such that for any $B\in \mathcal H$ we have $$\int_B X(\omega)d\mathbb P(\omega)=\int_B \EE_{\mathcal H}(X)(\omega)\mathrm{d}\mathbb P(\omega).$$

The subset relation of $\sigma$-algebras $\mathcal H_1 \subseteq \mathcal H_2$ yields the compatibility relation at the level of conditional expectations $\EE_{\mathcal H_1}\circ \EE_{\mathcal H_2}=\EE_{\mathcal H_1}$. In particular, for the smallest $\sigma$-algebra $\mathcal H_0=\{\emptyset, \Omega\}$ we obtain the usual expectation, whereas $\mathcal H=\mathcal F$ gives the identity. If $\Omega$ is finite, the collection of conditional expectations on $\Omega$ is in one-to-one correspondence with the poset of set partitions $\SP(|\Omega|)$ (see Def. \ref{def:partitions}).

(2). The algebra $\cA=M_n(\CC)$ of deterministic matrices, together with the normalized trace $\tau=\frac{1}{n}\mathrm{Tr}$ is a NCPS with adjoint given by the hermitian transpose. 

Before proceeding further, note that the symmetric Bernoulli random variable $X$ in the context of the first example $(\cA_{ \mathcal F},\EE)$, (i.e. a fair coin with outcomes $\pm1$) and the matrix $$B=\left(\begin{array}{cc}
0&1\\
1&0
\end{array}\right)\in (M_2(\CC),\frac{1}{2}\mathrm{Tr}),$$ have the same algebraic distribution. 

Indeed, $X=\bar X$, $\EE(X)=0$ and $X^2=1_{\cA_{ \mathcal F}}$, and similarly, $B=B^*$, $\frac{1}{2}\mathrm{Tr}(B)=0$ and $B^2=I_2=1_{M_2(\CC)}$. Hence
\begin{eqnarray}
\tau_1(X^{\varepsilon_1}X^{\varepsilon_2}\dots X^{\varepsilon_m})&=&\tau_1(X^m)\\
&=&\tau_2(B^m)\\
&=&\tau_2(B^{\varepsilon_1}B^{\varepsilon_2}\dots B^{\varepsilon_m}),
\end{eqnarray}
for all  $\varepsilon=(\varepsilon_1,\dots,\varepsilon_m)\in \{1,*\}^m$  and thus, all moments coincide.

Similarly, the unitary random variable $u_k\in \cA_{ \mathcal F}$ with uniform weights on the $k$-roots of $1$ has the same distribution as the unitary deterministic matrix  $U_k\in M_k(\CC)$ with ones above the diagonal. Their distribution is characterized by the moments, which yield, for $\varepsilon=(\varepsilon_1,\dots,\varepsilon_m)\in \{1,-1\}^k$ 
$$\tau(u_k^{\varepsilon_1}u_k^{\varepsilon_2}\dots u_k^{\varepsilon_m})=\left\{\begin{array}{ll}
1&\sum_i^m \varepsilon_i=0 \mod k\\
0& \text{otherwise.}
\end{array}\right.$$

As for general non-commutative random variables, a matrix $A\in M_n(\CC)$ is normal if $AA^*=A^*A$. Normality is equivalent to the geometric condition that $A$ and $A^*$ are simultaneously diagonalizable (i.e. $A=U\Lambda U^*$ and $A^*=U\Lambda^* U^*$) for some unitary matrix $U\in M_n(\CC)$ and a diagonal matrix $\Lambda\in M_n(\CC)$ containing the eigenvalues of $A$.

Hence, for a normal matrix $A$, the moments of $A,A^*$ are, for each $k\geq 1$ and $(\ee(1),\ee(2),\dots ,\ee(k))\in\{1,*\}^k$ with $\ee_1=|\{i\leq k:e(i)=1\}|$ and $\ee_*=|\{i\leq k:e(i)=*\}|$
\begin{eqnarray}
\frac{1}{n}\mathrm{Tr}(A^{\ee(1)}A^{\ee(2)}\cdots A^{\ee(k)})&=&\frac{1}{n}\mathrm{Tr}(U\Lambda^{\ee(1)}U^*U\Lambda^{\ee(2)}U^*\cdots U\Lambda^{\ee(1)}U^*) \\ 
&=&\frac{1}{n}\mathrm{Tr}(U\Lambda^{\ee(1)}\Lambda^{\ee(2)}\Lambda^{\ee(3)}\cdots \Lambda^{\ee(1)}U^*) \\ 
&=&\frac{1}{n}\mathrm{Tr}(\Lambda^{\ee_1}\Lambda^{\ee_*}) \\ 
&=&\frac{1}{n} \sum_{i\leq n} (\lambda_i)^{\ee_1} (\overline{\lambda_i})^{\ee_*},
\end{eqnarray}
which are exactly the moments (in the usual, probabilistic sense) of the complex random variable $X$ whose distribution $\mu_A$ assigns a weight of $\frac{1}{n}$ to each eigenvalue $\lambda_i$ of $A$.

In general, for a normal element $a$ in a non-commutative probability space $(\cA,\tau)$, we say that a probability measure $\mu:=\mu_a$ on the complex plane is the \emph{analytic distribution} of $a$ if $$\tau(a^m (a^*)^l)=\int_{\CC}z^m (\bar z)^l\mathrm{d}\mu(z),$$ for all $m,l\geq 0$. The existence of analytic distributions for normal elements is guaranteed if $\cA$ is a $C^*$-algebra. 

Canonical conditional expectations compatible with $\frac{1}{n}\mathrm{Tr}:M_n(\CC)\to \langle I_n \rangle$ are the identity $\mathrm{id}:M_n(\CC)\to M_n(\CC)$ and the projection to the diagonal $$\FF:(a_{ij})_{ij}\mapsto (a_{ij}\delta_{ij})_{ij}.$$

(3). Let $(\cA,\tau)$ be a NCPS with pairwise orthogonal projections $p_0,\dots, p_d\in \cA$, with $\tau(p_i)>0$ and $p_0+p_1+\dots+p_d=1_{\cA}$. Then there is a unique conditional expectation $\FF:\cA\to \langle p_0,\dots,p_d\rangle$ compatible with $\tau$, given by
$$\FF(a)=\sum_{i\leq k}\tau(p_i)^{-1}\tau(p_iap_i)p_i.$$ Such an OVPS is called a \emph{rectangular probability space} \cite{BG09,BG09a}.

For diagonal elements in a rectangular space  (i.e. elements $a\in \cA$ of the form $a=\sum_i p_iap_i$), the study of the distribution of $a$ is essentially the separated study of the pieces $a_i:=p_iap_i$. Indeed, the conditional expectation $\FF:\cA\to\cB$ splits as a sum $\FF=\sum_i \FF_i$, where $\FF_i:a\mapsto \tau(p_i)^{-1}\tau(p_iap_i) p_i$ is a \emph{compression} (i.e. the restriction $(p_i\cA p_i,\FF_i|_{p_i\cA p_i})$ is a scalar probability space with unit $p_i$).

Hence, it is natural to associate a normal diagonal element $a\in \cA$ in a rectangular space with the tuple of probability measures $\mu=(\mu_0,\mu_1,\dots, \mu_d)$, where $\mu_i$ is the analytic distribution of the normal element $p_iap_i$ in the compressed space $(p_i\cA p_i,\FF_i|_{p_i\cA p_i})$. 

Histograms can also be interpreted in terms of rectangular spaces (see Section 5). 

(4). We may endow the algebra $\cA=M_n(\CC)$ with the functional $\tau_{11}((a_{ij}))=a_{11}$ in order to obtain a more exotic (non-tracial) NCPS. Again, if $A\in M_n(\CC)$ is normal, then the moments with respect to the functional $\tau_{11}(*)=\langle e_{11},*e_{11}\rangle$ are the moments of a discrete measure $\mu$. 

Indeed, if the vector $e_1:=(1,0,0,\dots,0,0)$ is expressed as $e_1=\sum_{i\leq n}\alpha_i v_i$ in terms of the basis of eigenvectors $\{v_1,\dots, v_n\}$ of $A$, then clearly $$\tau_{11}(A^k(A^*)^l)=\langle e_{11},A^k(A^*)^le_{11}\rangle=\langle \sum_{i\leq n}\alpha_i v_i,\sum_{i\leq n}\lambda_i^k\bar \lambda_i^l\alpha_i v_i\rangle=\sum_{i\leq n}\lambda_i^k \bar \lambda_i^l\alpha_i^2,$$ which are the moments of the discrete probability measure with atoms at $\lambda_1,\lambda_2,\dots ,\lambda_n$ with corresponding weights $w_i=\alpha_i^2$. The fact that these are indeed the weights of a probability measure follows from the orthonormality of the eigenbasis.

It also follows trivially that, when fixing a vector of the eigenbasis and considering not only $\tau_{11}$ but running over all the possible $(\tau_{jj})_{j\leq n}$, the sum of the corresponding weights $w_i^{(j)}$ adds one, verifying that the average $\tau=\sum_i\frac{1}{n}\tau_{ii}=\frac{1}{n}\mathrm{Tr}$ (i.e., the normalized trace) assigns the discrete uniform measure on the eigenvalues (with multiplicities) as the distribution of a deterministic matrix.

This collection of analytic distributions is relevant in algebraic topology. For example, if the matrix $A$ happens to be an adjacency matrix of a graph, the weight $w_i^{(j)}$ is non-zero if and only if the vertex $v_j$ is contained in the same connected component as $v_i$. We will come back to more topological-algebraic considerations later, in Sections 3 and 4.

(5). If $(\cA_1,\cB_1, \EE_1), (\cA_2,\cB_2, \EE_2)$ are two OVPS, their tensor product $(\cA_1\otimes \cA_2,\cB_1\otimes \cB_2, \EE_1 \otimes \EE_2)$ is also an OVPS. If the spaces are actually scalar-valued, then the tensor product is also scalar-valued.

In particular the tensor product of two NCPS from examples (1) and (2) yields the NCPS $(M_n(\CC)\otimes\mathcal A_{\mathcal F},\frac{1}{n}\mathrm{Tr}\otimes \EE)$, which can be identified with the space of matrices with entries in $\cA_{\mathcal F}$ (\emph{random} matrices). Actually, for a normal random matrix $A\in M_n(\CC)\otimes \cA_{\mathcal F}$, the moments with respect to $\frac{1}{n}\mathrm{Tr}\otimes \EE$ coincide with the moments of the averaged spectral measure (i.e. the distribution of a random eigenvalue chosen from a random matrix).

Another operator-valued structure which has been useful in random matrix theory is $(M_n(\CC),\mathrm{id})\otimes (\mathcal A, \tau)=(M_n(\CC)\otimes\mathcal A,\mathrm{id} \otimes \tau)$, where expectation is taken entry-wise and no further matricial operation is performed. Free independence (see next definition and remark) is preserved by tensor products. 

\end{example}

Now that we have introduced non-commutative random variables, their moments and distributions, new notions of independence may be defined. 

\subsection{Non-commutative notions of stochastic independence and cumulants} A notion of independence must be understood as a rule for computing mixed moments.

In \cite{BGS02} an axiomatic approach is presented in terms of co-products of different categories of algebras. The new non-commutative notions of stochastic independence are the \emph{universal} non-commutative generalizations of the classical (commutative) notion of independence.

The axioms on the factorizations of mixed moments (associativity, symmetry, universality), have been formulated algebraically in terms of universal properties. Thus, each notion of independence leads to a robust, parallel non-commutative probability theory, where not only fundamental theorems (e.g. laws of large numbers, central limit theorems, Poisson limit theorems, etc.) can be formulated, but much more sophisticated concepts, such as L\'{e}vy processes or an analytic theory of characteristic functions, have been considered and developed. 
 
The axioms are also consistent with our usual notion of independence, in the sense that they collapse when restricted to categories of commutative algebras, leaving the tensor product (an hence, classical independence) as the unique \emph{commutative} notion of independence. It turns out that, in the general non-commutative setting, there are only three universal (and five natural) notions of independence:

\begin{definition}\label{def:ncindeps}
Let $(\cA,\cB, \FF)$ be a $\cB$-probability space. Let $A_1,\dots,A_k \subseteq \cA$ be algebras containing $\cB$.

(1). The algebras $A_1,\dots,A_k$ are $\cB$-classically independent if they commute and 
$$\FF(a_1^{n_1}a_2^{n_2}\cdots a_k^{n_k})=\FF(a_1^{n_1})\FF(a_2^{n_2})\cdots \FF(a_k^{n_k}),$$ for any $a_i\in A_i$, $n_i\geq 0$, $i\leq k$. 

(2). For any $a\in \cA$, let $\dot a:=a-\FF(a)$. The algebras $A_1,\dots,A_k$ are $\cB$-free iff $$\FF(\dot a_1 \dot a_2\dots  \dot a_m)=0,$$ for any $m\geq 1$ and any $a_i\in A_{j(i)}$, $i\leq m$, such that $j(1)\neq j(2)\neq\dots \neq j(m)$.

(3). The algebras $A_1,\dots,A_k$ are $\cB$-Boolean independent if
$$\FF(a_1a_2\cdots a_m)=\FF(a_1)\FF(a_2)\cdots \FF(a_m),$$
for any $m\geq 1$ and any $a_i\in A_{j(i)}$, $i\leq m$, such that $j(1)\neq j(2)\neq\dots \neq j(m)$.
\end{definition}

\begin{remark}\label{rem:indeps}
(1). The monotone independence and its mirror-image, the anti-monotone independence are non-symmetric notions of independence. That is, $A_1,A_2$ being independent does not imply that $A_2,A_1$ are independent. In order to avoid cumbersome definitions, we will introduce these notions of independence after some notation on partitions and multiplicative maps.

(2). The notions of independence are quite exclusive (at least in the scalar case $\cB=\CC$). For example, if two self-adjoint random variables are classically independent and free at the same time, then at least one of the variables must be a constant. In the free and classical case, constants are trivially independent from the whole algebra $\cA$. In the operator valued-level, the algebra of constants $\langle 1_{\cA} \rangle$ is essentially replaced by a larger algebra $\cB\subseteq\cA$. 

The notions of independence become less exclusive as the algebra $\cB\subseteq \cA$ grows larger. For example, random variables which are orthogonal to each other can be thought as being independent with respect to some rectangular space.

Indeed, let $(\cA,\tau)$ be a rectangular probability space with $\cB=\langle p_0,\dots, p_d \rangle$. Then the pieces $p_0ap_0,\dots, p_dap_d$ of a diagonal element are \emph{classically, free and Boolean independent} with respect to the conditional expectation $\FF:\cA\to\cB$.

This allows to express convex combinations of probability measures as sums of independent random variables. The extreme case $\cB=\cA$ implies $\FF=\mathrm{id}_{\cA}$, which makes all variables in $\cA$ independent (in all three senses).

(3). Probably one of the simplest and most useful properties of free independence is that it is preserved when operating with tensor products: If $A_1,A_2\subseteq \cA$ are $\cB$-free in $(\cA,\cB,\FF)$, then $A_1\otimes \cB',A_2\otimes \cB'\subseteq \cA \otimes \cB'$ are $\cB \otimes \cB'$-free in the OVPS $(\cA \otimes \cB',\cB \otimes \cB' ,\FF\otimes \mathrm{id})$.

In particular two matrices filled up with entries coming from free algebras are free over the algebra of complex matrices (case $\cB'=M_n(\CC), \cB=\CC$).
\end{remark}

A couple of decades ago, Speicher introduced  in \cite{Spe94} a combinatorial characterization of freeness which relies on Rota's works on incidence algebras of posets and multiplicative functions (see \cite{NS06}, for a comprehensive course, and \cite{Spe98} for the operator-valued case). A key concept for this combinatorial approach to independence are the cumulants. In order to define these, we must first introduce some notation on set partitions.

\begin{definition}\label{def:partitions}
We denote by $\mathcal P (n)$ the partitions of the set $[n]=\{1,\dots, n\}$. A partition $\pi=\{V_1,V_2,\dots,V_k\}\in \SP(n)$ is a decomposition of $[n]=\bigcup_{\i\leq k} V_i$ into disjoint, non-empty subsets (called blocks). $\mathcal P(n)$ is a poset with respect to the order of reverse refinement, with minimum $0_n:=\{\{1\},\{2\},\{3\},\dots, \{n\}\}$ and maximum $1_n:=\{1,2,3,\dots,n\}$.

Any partition $\pi$ defines an equivalence relation and hence we may write $a\sim_{\pi} b$ (or simply $a\sim b$) meaning that $a,b\in V$ for some block $V\in \pi$. Sometimes we abbreviate notation and write $\{1346,28,57\}$ for $\pi=\{\{1,3,4,6\},\{2,8\},\{5,7\}\}$. 

A partition $\pi=\{V_1,\dots ,V_k\}\in\SP(n)$ is a non-crossing partition if it fulfills the condition, that for any quadruple $(1\leq a<b<c<d\leq n)$, $a,c\in V_i\in \pi$ and $b,d\in V_j\in \pi$ implies $V_i=V_j$. 

A partition $\pi=\{V_1,\dots ,V_k\}\in\SP(n)$ is an interval partition if for any triple $(1\leq a<b<c\leq n)$, $a,c\in V_i\in \pi$ and $b\in V_j\in \pi$ implies $V_i=V_j$.

We denote respectively by $\IP(n)\subseteq \NCP(n) \subseteq \SP(n)$, the sub-posets of interval partitions and non-crossing partitions. 
\end{definition}

In the context of an operator-valued probability space $(\cA,\cB)$, we will often consider sequences of $\cB$-multi-linear maps $f=(f_n)_{n\geq 1}$, $f_n:\cA^n\to\cB$. The sequence $f$ may be extended multiplicatively to families of partitions $(f_{\pi})_{\pi\in L(n)\subseteq \SP(n)}$ in a natural way.

For the case of non-crossing partitions, $\pi \in L(n)\subseteq \NCP(n)$, the multiplicative extension of $f$ may be defined in the general context of a non-commutative algebra $\cA$.

For $\pi\in \NCP(n)$ we define the $\cB$-multi-linear map $f_{\pi}:\cA^n\to\cB$ which evaluates $f$ in a nested way, according to the partition $\pi$. For example, for $\pi=\{1349,2,57,6,8\}$ we have $f_{\pi}:\cA^9\to\cB$, given by
$$f_{\pi}(a_1,a_2,\dots,a_9)=f_4(a_1f_1(a_2),a_3,a_4f_2(a_5f_1(a_6),a_7)f_1(a_8),a_9)$$

Very often we ask our sequence of functionals to be $\cB$-balanced (i.e. for any $n\geq 1$ and $i\leq n-1$, $$f_n(a_1,a_2,\dots,a_i,ba_{i+1},\dots,a_n)=f_n(a_1,a_2,\dots,a_ib,a_{i+1},\dots,a_n),$$ so that the definition of $f_{\pi}$ does not depend on the decision of multiplying the nested map on the left or the right argument.  

In the context of an operator-valued probability space $(\cA,\cB,\FF)$, the fundamental multiplicative family of $\cB$-balanced functionals that we use are the moment multi-linear map $$\FF_n:(a_1,a_2,\dots,a_n)=\FF(a_1a_2a_3\cdots a_n).$$

Now we are able to define the monotone and anti-monotone notions of independence.

\begin{definition}
The sub-algebras $A_1,\dots,A_k\subseteq \cA$ are monotone independent (in that order) iff the mixed moments of $\langle A_1,\dots,A_k \rangle$ are computed according to the following rule:

For all $m\geq 1$ and any mixed moment $a_1a_2\dots a_m$, with $a_i\in A_{j(i)}$, $i\leq m$ we have $$\FF(a_1a_2\dots a_m)=\FF_{\pi}(a_1,a_2,\dots ,a_m),$$
where $\pi\in\NCP(m)$ is a non-crossing partition, constructed from the indices $(j(i))_{i\leq m}$ as follows:

We first let $W_1=\{i: j(i)=1\}$ be the first block of $\pi$. We then decompose the subset $W_{2}=\{i:j(i)=2\}$ into the blocks of $\pi$, by joining as many of those indices as possible, without crossing the previously selected block $W_1$. Now that the blocks containing the elements from the algebras $A_{1},A_2$ have been fixed, we decompose $W_3=\{i:j(i)=3\}$ into the blocks of $\pi$ by joining indices to form a maximum non-crossing partition, and so on, until all the blocks of the partition $\pi$ are determined.

For example, let $A_1,A_2,A_3$ be monotone independent and let $a_i\in A_i$. We associate the mixed moment $a_1a_2a_3a_2a_1a_3a_1a_2a_3$, with the partition $\pi=\{157,24,8,3,6,9\}$, and hence $$\FF(a_1a_2a_3a_2a_1a_3a_1a_2a_3)=\FF(a_1\FF(a_2\FF(a_3)a_4)a_5\FF(a_6)a_7)\FF(a_8)\FF(a_9)$$

The anti-monotone independence is just the monotone independence in the reverse order, in particular $A_1,\dots,A_k$ are monotone independent if and only if $A_k,\dots ,A_1$ are anti-monotone independent, and thus, the study of the later case is often omitted.
\end{definition}

There is a problem for defining $f_{\pi}$ for a crossing partition $\pi$ in the general non-commutative situation. However, if we assume $\cA$ to be a commutative algebra, we may simply ignore nestings and crossings and define, for example: $$f_{\{13,24\}}(a_1,a_2,a_3,a_4)=f_2(a_1,a_3)f_2(a_2,a_4).$$

In these contexts cumulants may be defined to characterize the afore mentioned notions of independence. The following recursions define the different cumulant maps. 

\begin{definition}\cite{Spe94,SW97,HS11}.
The free, Boolean, monotone and classical cumulant maps $(R_n)_{n\geq 1}$, $(B_n)_{n\geq 1}$, $(H_n)_{n\geq 1}$ and $(K_n)_{n\geq 1}$, $$R_n,B_n,H_n,K_n:\cA^n\to\cB,\quad n\geq 1,$$ are defined recursively through the moment-cumulant formulas:
\begin{eqnarray}
\FF_n(a_1,\dots,a_n)&=:&\sum_{\pi\in \NCP(n)}R_{\pi}(a_1,\dots,a_n).\\
&=:&\sum_{\pi\in \IP(n)}B_{\pi}(a_1,\dots,a_n).\\
&=:&\sum_{\pi\in \NCP(n)}\frac{1}{(\mathrm{nest}(\pi))!}H_{\pi}(a_1,\dots,a_n).
\end{eqnarray}
See \cite{AHLV15} for the definitions of the \emph{nesting tree} $\mathrm{nest}(\pi)$ of a non-crossing partition and its \emph{tree factorial}. For the case in which $\cA$ is commutative, the classical cumulants are defined accordingly, as the solutions to the recursion
$$\FF_n(a_1,\dots,a_n)=:\sum_{\pi\in \SP(n)}K_{\pi}(a_1,\dots,a_n).$$ 
\end{definition}

The cumulants may be solved in terms of the moments. For example, since $\IP(n)=\NCP(n)=\SP(n)$ for $n=1,2$ we get that the first cumulant is simply the mean
$$R_1(a)=B_1(a)=H_1(a)=K_1(a)=\FF(a),$$ whereas the second cumulant is the covariance $$R_2(a_1,a_2)=B_2(a_1,a_2)=H_2(a_1,a_2)=K_2(a_1,a_2)=\FF(a_1a_2)-\FF(a_1)\FF(a_2).$$

Cumulants maps may be defined in terms of moment maps by performing a M\"obius inversion with respect to the corresponding lattice of partitions. That is, we have the inverse, cumulant-moment formulas:
\begin{eqnarray}
K_n(a_1,\dots,a_n)&=&\sum_{\pi\in \SP(n)}\FF_{\pi}(a_1,\dots,a_n)\mu_{\SP}([\pi,1_n]),\\
R_n(a_1,\dots,a_n)&=&\sum_{\pi\in \NCP(n)}\FF_{\pi}(a_1,\dots,a_n)\mu_{\NCP}([\pi,1_n]),\\
B_n(a_1,\dots,a_n)&=&\sum_{\pi\in \IP(n)}\FF_{\pi}(a_1,\dots,a_n)\mu_{\IP}([\pi,1_n]),
\end{eqnarray}
where, $\mu_{\SP},\mu_{\NCP},\mu_{\IP}:\SP^2\to \CC$ are the M\"obius functions (defined as the inverses of the zeta functions in the framework of incidence algebras).

\begin{example}[Special non-commutative random variables]

Let $(\cA,\FF)$ be some $\cB$-probability space.

(1). Constant random variables $b\in \cB$ have cumulants, for $m\geq 1$ and $\varepsilon=(\varepsilon_1, \varepsilon_2,\dots,\varepsilon_m)\in\{1,*\}^m$:
$$C_m(b^{\varepsilon_1},b^{\varepsilon_2},\dots, b^{\varepsilon_m})=\left\{
\begin{array}{cc}
b^{\varepsilon_1},& m=1,\\
0,& \text{otherwise},\end{array}
\right.$$
where, $(C_n)_{n\geq 1}\in\{(K_n)_{n\geq 1}, (R_n)_{n\geq 1}, (B_n)_{n\geq 1}\}$ are any of the cumulant maps.

Now let $(\cA,\tau)$ be a scalar-non-commutative probability space.

(2). A (normal\footnote{in the sense $zz^*=z^*z$.}) random variable $z$ is standard complex gaussian if it has the same law as $\frac{1}{\sqrt 2}(x+iy)$, where $x=x^*,y=y^*\in \cA$ are independent standard gaussian random variables. 

Equivalently, in terms of cumulants (eq. moments), a random variable is complex gaussian if the mixed classical cumulants on $z,z^*$ are given, for $m\geq 1$ and $\varepsilon=(\varepsilon_1, \varepsilon_2,\dots,\varepsilon_m)\in\{1,*\}^m$ by 
$$K_m(z^{\varepsilon_1},z^{\varepsilon_2},\dots, z^{\varepsilon_m})=\left\{
\begin{array}{cc}
1,& m=2, \varepsilon_1\neq \varepsilon_2\\
0,& \text{otherwise}.\end{array}
\right.$$
Observe that $\frac{1}{\sqrt 2}(z+z^*)$ is (real) standard gaussianl.

(3). A (non-normal!) random variable $c\in \cA$ is a standard circular element if it has the non-commutative distribution of $\frac{1}{\sqrt 2}(s_1+is_2)$, where $s_1=s_1^*,s_2=s_2^*\in \cA$ are free standard semicircular elements. A circular element is characterized by its free cumulants
$$R_m(c^{\varepsilon_1},c^{\varepsilon_2},\dots, c^{\varepsilon_m})=\left\{
\begin{array}{cc}
1,& m=2, \varepsilon_1\neq \varepsilon_2\\
0,& \text{otherwise}.\end{array}
\right.$$
Semicircular elements are self-adjoint random variables $s=s^*$ such that
$$R_m(s,s,\dots, s)=\left\{
\begin{array}{cc}
1,& m=2, \\
0,& \text{otherwise}.\end{array}
\right.$$
Observe that $\frac{1}{\sqrt 2}(c+c^*)$ is a standard semicircular element.

(4). A (normal) random variable $u\in \cA$ is a Haar-distributed unitary random variable if it is unitary (i.e. $u^*=u^{-1}$) and its moments are given for $\varepsilon=(\varepsilon_1, \varepsilon_2,\dots,\varepsilon_m)\in\{1,-1\}^m$ by
$$\tau(u^{\varepsilon_1},u^{\varepsilon_2},\dots, u^{\varepsilon_m})=\left\{
\begin{array}{cc}
1,& if \sum_i\varepsilon_i=0\\
0,& \text{otherwise}.\end{array}
\right.$$
\end{example}

From probability theory, we know that tensor independent random variables have vanishing covariance. The converse is not true. This is not a surprise, as the vanishing of the covariance can be determined by simply knowing mixed moments of first and second orders, whereas, in view of Def. \ref{def:ncindeps}, independence is equivalent to an infinite number of equations on the moments. 

The fact that all these factorizations hold for the left-hand side of the moment-cumulant formulas turns out to be equivalent to the more uniform condition, on the right hand side, that all mixed cumulants vanish. Since moments and cumulants determine each other, cumulants are the higher dimensional generalizations of the covariance that we require to characterize independence.

Let $\cB\subseteq A_1,\dots,A_k \subseteq \cA$ be sub-algebras in some $\cB$-probability space $(\cA,\cB,\FF)$. For any choice of cumulant maps $(C_n)_{n\geq 1}\in\{(K_n)_{n\geq 1}, (R_n)_{n\geq 1}, (B_n)_{n\geq 1}\}$, we call
$C_m(a_1,a_2,\dots,a_m)$, $a_i\in A_{j(i)}$, $i\leq m$ a \emph{mixed cumulant} if there are $1\leq r<s\leq m$ such that $j(r)\neq j(s)$. 

For example, there are no mixed cumulants of order $1$ and the covariance $C_2(a_1,a_2)$ is a mixed cumulant if $j(1)\neq j(2)$.

\begin{theorem} \cite{Spe94,SW97}. Let $\cB\subseteq A_1,\dots,A_k \subseteq \cA$ be sub-algebras in some $\cB$-probability space $(\cA,\cB,\FF)$.

(1). The algebras $A_1,\dots,A_k$ are $\cB$-classically independent iff $\langle A_1,\dots,A_k \rangle$ is commutative and all the mixed classical cumulants vanish.   

(2). The algebras $A_1,\dots,A_k$ are $\cB$-free iff all the mixed free cumulants vanish.   

(3). The algebras $A_1,\dots,A_k$ are $\cB$-Boolean independent iff all the mixed Boolean cumulants vanish.   

\end{theorem}

Combined with multi-linearity (which follows directly from the multi-linearity of $\FF_n$), the previous characterization implies that cumulants are additive for independent random variables. For example, if $(X_N)_{n\geq 1}$ are self-adjoint i.i.d random variables and $S_N:=X_1+X_2+\dots+X_N$, then \begin{equation}\label{eq:cumadd}
K_m(S_N,S_N\dots,S_N)=\sum_{i\leq N} K_m(X_i,\dots,X_i)=N K_m(X_1,\dots,X_1),
\end{equation}
since all the mixed cumulants vanish.

Hence, the Law of large lumbers follows: $$\lim_{N\to \infty}K_m(N^{-1}S_N, N^{-1}S_N,\dots ,N^{-1}S_N)= \left\{
\begin{array}{cc}
\EE(X_1),& m=1\\
0,& m\geq 2,\end{array}
\right.$$
because these are exactly the cumulants of a constant random variable $c=\EE(X_1)$.  For \emph{centered}, self-adjoint i.i.d random variables $(X_N)_{n\geq 1}$, we obtain the central limit theorem in the same way: $$\lim_{N\to \infty}K_m(N^{-1/2}S_N, N^{-1/2}S_N,\dots ,N^{-1/2}S_N)= \left\{
\begin{array}{cc}
\EE(X_1)=0,& m=1,\\
\EE(X_1^2)=\sigma^2,& m=2,\\
0,& m\geq 3,\end{array}
\right.$$
since the vanishing of all cumulants or order greater than $3$ characterizes normal random variables. Thus, cumulants enable direct proofs for the weak versions of the LLN and the CLT and the exact same arguments can be used to derive the free, Boolean and monotone\footnote{Although no general characterization of monotone independence is yet available in terms of orthogonality of monotone cumulants, they do satisfy the additive property of Eq. (\ref{eq:cumadd}), when restricted to \emph{identically distributed} random variables, as above.} versions. The corresponding Gaussian distributions are, respectively, Wigner's semicircle law, the symmetric Bernoulli distribution and the arcsine distribution.

\section{Elementary realizations of independent non-commutative random variables}

The spectral graph theory is at the heart of non-commutative probability theory; the simplest realizations of independent non-commutative random variables and their convolutions are exactly through graphs and their spectra.

The other prominent realization is through Voiculescu's asymptotic freeness of random matrices. Although, technically speaking, free independence is not yet directly related to algebraic topology there are some similarities between the main concepts featuring in TDA and random matrix theory. Free probability is, at the moment, the most developed non-commutative probability theory and the one with the most impressive algorithmic implementations, which we will sketch later in Section 6. 

Boolean probability theory has proved to be useful, at least indirectly, for understanding problems related to classical and (specially) free probability theories. As examples, see the homomorphism in \cite{BN08a}, which explains the presence of Boolean transforms in the algorithms for computing free multiplicative convolutions \cite{BB07,BSTV14}, or the proof of fourth moment theorem for infinitely divisible random variables \cite{Ari13}, which relies on the continuity of the so-called Bercovici-Pata bijections between infinitely divisible random variables \cite{BP99,BN08b}.

The realization of the additive Boolean convolution in terms of graphs is one of the main motivations for our new notions of distributions for simplicial complexes in Section 4, with strong connections to algebraic topology. A first instance of such kind of connection was observed in \cite{DPT15a}.

\subsection{Tensor products and classical independence}

In algebraic probability we sometimes refer to the classical independence as tensor independence, as it can be realized through tensor products in the framework of Example 
\ref{ex:ovps} (2). 

Indeed if $A\in M_n(\CC)$ and $B\in M_m(\CC)$ are normal matrices, then so are $\tilde A:=A\otimes I_m$, $\tilde B:=I_n\otimes B,$ $\tilde A +\tilde B$ and $\tilde A \tilde B=A\otimes B$. Furthermore, if $A=UD_1U^*$  and $B=VD_2V^*$ are their Jordan normal forms, then $U\otimes V$ simultaneously diagonalizes $\tilde A$, $\tilde B$, $\tilde A +\tilde B$ and $\tilde A \tilde B$  which are similar, respectively, to the diagonal matrices $D_1\otimes I_m$, $I_n\otimes D_2$, $D_1\otimes I_m+I_n\otimes D_2$ and  $D_1\otimes D_2$. Observe that $\mu_A=\mu_{\tilde A}$ and $\mu_B=\mu_{\tilde B}$.

In addition, $D_1\otimes I_m+I_n\otimes D_2$ (resp. $D_1\otimes D_2$) are the diagonal matrices with exactly all the possible sums (resp. products) of pairs of eigenvalues. Hence, at the level of spectral measures $\tilde A$, $\tilde B$ are \emph{realizations} of independent random variables with given distributions and thus, independence may be understood algebraically as a special positioning of operators.  

From an algebraic-probabilistic point of view (in terms of moments), it would have been enough for us to simply observe that the unital algebra generated by $\langle \tilde A,\tilde A^*,\tilde B,\tilde B^*\rangle $ commutes, and that $\frac{1}{nm}\mathrm{Tr}=\frac{1}{n}\mathrm{Tr}\otimes \frac{1}{m}\mathrm{Tr}$, and hence the mixed moments of $\tilde A,\tilde A^*$ and $\tilde B,\tilde B^*$ are computed according to the factorization given by classical independence.

\subsection{Graph products and non-commutative independence}\label{subsection:graphprod}

The tensor product yields independent random variables, even if we we consider the alternative, non-tracial NCPS $(M_n(\CC),\tau_{11})$ (for this, we need only to observe that, again, the projection to the upper-left corner $\tau_{11}:M_{nm}(\CC)\to \CC$ is the tensor product of the corresponding projections to the upper corners $\tau^{(1)}_{11}:M_n(\CC)\to\CC$ and $\tau^{(2)}_{11}:M_m(\CC)\to\CC$). 

The situation gets even nicer if the matrices in question are adjacency matrices of graphs.  Let $G=(V,E)$ be a graph. If $|V|=n$, we define the adjacency matrix of $G$  as the matrix $A_G=(a_{ij})\in M_n(\CC)$ with entries 
$$a_{ij}=a_{ji}=\left\{
\begin{array}{cc}
1,& (v_i,v_j)\in E,\\
0,& \text{otherwise}.\end{array}
\right.$$

First of all, observe that if $A,B$ are (necessarily self-adjoint) adjacency matrices of graphs, so are the matrices $A\otimes I_m$, $I_n\otimes B$, and $A\otimes I_m+I_n\otimes B$.

The moments of $A_{G}$ w.r.t. $\frac{1}{n}\mathrm{Tr}$ are of course still the moments of the spectral measure, but now have a combinatorial interpretation:
$$\frac{1}{n}\mathrm{Tr}(A_G^k)=\frac{1}{n}\sum_{i:[k]\to[n]} a_{i(1)i(2)}a_{i(2)i(3)}\dots a_{i(k)i(1)},$$
and thus the $k$-th moment of $A_G$ counts cycles (i.e. sequences of adjacent edges $(v_{i(1)},v_{i(2)})(v_{i(2)},v_{i(3)}),\dots,(v_{i(k)},v_{i(1)})$) of size $k$ within the graph. 

Switching from the functional $\frac{1}{n}\mathrm{Tr}$ to $\tau_{11}$ can be interpreted nicely in terms of moments. By considering moments with respect to the functional $\tau_{11}$, the vertex $v_1\in V$ (corresponding to the first row/column of the adjacency matrix) plays now a special role in the graph. We should think that the graph is \emph{rooted} at the vertex $v_1$. The moments now count the number of $k$-cycles which start and end at the root
$$\tau_{11}(A_G^k)=\sum_{i:[k-1]\to[n]} a_{1i(1)}a_{i(1)i(2)}\dots a_{i(k-1)1}.$$
It is not hard to show by induction that the Boolean cumulants count cycles at the root which are irreducible in the sense that the cycle does not visit the root at any intermediate step. In other words, for a cycle $(v_1,v_{i(1)})(v_{i(1)},v_{i(2)}),\dots,(v_{i(k-1)},v_{1})$ in $G$ we have
$$B_k(a_{1i(1)},a_{i(2)i(3)},\dots ,a_{i(k-1)1})=\left\{
\begin{array}{cc}
1,& \text{if } i(1),i(2),\dots i(k-1)\neq 1,\\
0,& \text{otherwise},\end{array}
\right.$$

We have seen that it is possible to realize tensor independent random variables w.r.t. $\tau_{11}$. An even more interesting fact is that, with respect to $\tau_{11}$, it is possible to realize  Boolean, monotone and free convolutions in terms of appropriate graph products \cite{ABO04,ALS07,Oba08}. We only illustrate the Boolean case below, as it is simpler and more relevant for our purposes.

For two rooted graphs $G_1$ and $G_2$, consider their \emph{star product} $G_1\star G_2$ obtained by identifying the two graphs by their roots, and considering the identified vertex as the new root. From our discussion above, it is clear that the Boolean cumulants of $G_1\star G_2$ are the sum of the Boolean cumulants of $G_1$ and $G_2$: any irreducible cycle of $G_1\star G_2$ is either an irreducible cycle in $G_1$ or $G_2$, as any cycle including edges from $G_1$ and $G_2$ needs to cross the root at some intermediate point. Thus, $G_1\star G_2$ realizes the additive Boolean convolution of the probability measures. 

Hence Boolean cumulants filter cycles within graphs according to some local topological property. Coincidentally, other examples of functionals which are additive with respect to the star product of graphs $G_1\star G_2$, or more general ``star'' products of topological spaces, are the Betti numbers (see Section 4).  

Unfortunately, the scalar-valued classical and free cumulants do not provide additional information about cycles. Indeed, the space $(M_n(\CC),\tau_{11})$ is so pathological, that the Boolean cumulants coincide with the other cumulant functionals when evaluated in elementary cycles. 

Hence, if we would like to filter cycles by more interesting topological features (e.g. a cycle being simple, a path being Hamiltonian, two cycles being homologous, etc.) by the means of cumulants, we may still need some more combinatorial ingredients. Our guess is that such cumulants will be found in the more general frameworks of non-commutative probability theory in the sense of traffics or homotopy probability theory.

Graphs are one dimensional skeletons of more general simplicial complexes. In the next section we investigate the problem of expressing topological invariants of simplicial complexes in terms of their non-commutative distributions. 

\section{Non-commutative distributions for simplicial complexes} \label{section:NCDSC}

We have seen that graphs are naturally considered as non-commutative random variables and graph products produce realizations of independent non-commutative random variables and their convolutions. It is natural then to ask if we may regard more general topological spaces, such as simplicial complexes, as non-commutative random variables. 

The usual notions of non-commutative distributions for graphs are weighted versions of the (discrete) spectral probability distributions of its adjacency matrix $A^0_1(G):=A_G$, and the moments may be described combinatorially by counting certain paths along adjacent edges on the graphs. These combinatorial interpretations of moments can be dualized and generalized to higher dimensions, by considering simplicial complexes instead of graphs, and moments described by sequences of $k$-dimensional faces, regarded as ``adjacent'' if they share a face of some fixed dimension $r\geq 0$. Thus, a notion of distribution for a simplicial complex could in principle be related to some multivariate combination of these generalized adjacency matrices $(A^r_k(X))_{0\leq r \neq k \leq d}$. We observe that, actually, much of this information can be organized very neatly in terms of operator-valued non-commutative probability. 

Let $X=\bigcup_{r\leq d}\bigcup_{j\leq n_r} (\sigma^{(r)}_j)$ be a finite simplicial complex of dimension $d$, where each $\sigma^{(r)}_j\in X$ is an $r$-dimensional face. We usually label the vertices with the numbers $1,2,\dots, n_0$ and denote a $j$-dimensional face $f=(i(0)i(1)\dots i(j))$ by listing the vertices it contains in increasing order. For example, we denote the abstract hollow tetrahedron by 
$$X=\{1,2,3,4,12,13,14,23,24,34,123,124,134,234\}.$$

Let $N=n_0+n_1+\dots+n_d$. Consider the vector space $\CC^N$ with a basis indexed by the faces of $X$ in increasing order, and consider its natural decomposition into subspaces generate by faces of a fixed dimension $\CC^N=\bigoplus_{0\leq i\leq d}\CC^{n_i}=\bigoplus_{0\leq i\leq d}C_i$.  The collection of orthogonal projections $(P_0,P_1,\dots,P_d)$ onto the chain groups $(C_0,C_1,\dots,C_d),$ defines a rectangular space in $M_N(\CC)$, with $\cB=\langle P_0,\dots, P_d\rangle$.  

We consider the incidence matrix $I(X)\in M_N(\CC)$ as an element of the operator-valued space $(M_N(\CC),\langle P_0, \dots, P_d \rangle,\FF)$. For $d=1$, $X=G$ is a graph and $I(X)^2=0$, so the $\langle P_0,P_1 \rangle$-distribution of $I(X)\in M_{n_0+n_1}(\CC)$ is very accessible. In fact, it is determined by the orthogonal operators $$I(X)^*I(X)=\Lambda^0_1(G)+A^0_1(G)\in M_{n_0}(\CC),\quad I(X)^*I(X)=\Lambda^1_0(G)+A^1_0(G)\in M_{n_1}(\CC),$$ where $A^0_1(G),A^1_0(G)=A^0_1(G')$ are the adjacency matrices of $G$ and its dual graph $G'$ and $\Lambda^1_0(G), \Lambda^0_1(G)$ are their (diagonal) degree matrices. 

If $G$ is a simple graph, $G'$ is $2$-regular, $\Lambda^0_1(G)=2I_{n_1}$, so the spectrum of $A^0_1{G'}$ can be obtained directly from $I(X)^*I(X)$ and thus from the $\langle P_0,P_1 \rangle$-distribution of $I(X)$. Therefore, the study of the non-commutative distribution of $G$ through the $\langle P_0,P_1 \rangle$-distribution of $I(G)$ (which can be also considered for arbitrary dimensions) is similar to its usual study in terms of $A^0_1(G)$ (in fact, the former contains the later for $G$ regular).

Moreover, for general $d\geq 1$ we obtain again the relations, for all $0\leq i\leq d$, $$P_iI(X)^*I(X)P_i=\Lambda^i_{i+1}(X)+A^i_{i+1}(X),\quad  P_iI(X)I(X)^*P_i=\Lambda^i_{i-1}(X)+A^{i}_{i-1}(X),$$ which involve higher dimensional adjacency/degree matrices. Simplicial complexes are the higher dimensional generalizations of simple graphs and we get that $\Lambda^i_{i-1}(X)=(i+2)I_{n_i}$.

For $d\geq 2$ it is no longer true that $I(X)^2=0$ and thus its distribution is harder to compute, but there is a way around this problem. Any order of the $n_0$ vertices of $X$ induces signed version $J(X)$ of the incidence matrix, called the boundary matrix, which plays an important role in algebraic topology. 

The boundary operator of a simplicial complex is defined as 
$$J(X)=P_0J(X)P_1+P_1J(X)P_2+\dots+P_{d-1}J(X)P_d\in M_N(\CC),$$
where each $P_{r-1}J(X)P_{r}$, restricted to $C_{r}$, is the $r$-dimensional boundary operator $\partial_r:C_{r}\to C_{r-1}$, which can be defined on a $r$-dimensional face $f=(i(0)i(1)\dots i(r))$ as $$\partial_r(f)=
\sum_{k\leq r}(-1)^{k}(i(0)i(1)\cdots i(k-1)\widehat{i(k)}i(k+1)\cdots i(r))$$
where $\widehat{i(k)}$ denotes that the vertex $i(k)$ has been removed from $f$ to obtain an $(r-1)$-dimensional face.

For example, if $X$ is the void tetrahedron, the boundary matrix is a $(4,6,4)$-block matrix
$$J(X)=\left(
\begin{array}{ccc}
0&\partial_1 & 0\\
0&0 & \partial_2\\
0&0 & 0
\end{array}\right)$$
with
$$\partial_1=
\left(
\begin{array}{cccccc}
1&1&1&0&0&0\\
-1&0&0&1&1&0\\
0&-1&0&-1&0&1\\
0&0&-1&0&-1&-1
\end{array}\right), \quad \partial_2=
\left(
\begin{array}{cccc}
1&1&0&0\\
-1&0&1&0\\
0&-1&-1&0\\
1&0&0&1\\
0&1&0&-1\\
0&0&1&1
\end{array}\right).
$$

By definition of the boundary matrix, one has that $P_{i-1}JP_{j}=\delta_{ij}\partial_i$ is either zero or the $i$-dimensional boundary operator $\partial_i:C^i\to C^{i-1}$. Thus, the fundamental equation $\partial_{i-1}\partial_{i}=0$ translates into the desired equation $J(X)^2=0$. 

Thus, we propose this $\cB$-valued distribution as \emph{the} non-commutative distribution for the simplicial complex $X$. First let us show that, for the case of a graph $X=G$, the distribution of $\langle J,J^*\rangle$ is still quite related to the usual distribution of the adjacency matrix $A_G=A_1^0(G)$.

\subsection{Graphs} If we let $d=1$, our simplicial complex is a graph $G$ and hence $\langle J,J^*\rangle$ depends simply on the joint distribution of the orthogonal operators $\partial_1\partial_1^*=P_0 JJ^* P_0=JJ^*$ and $\partial_1^*\partial_1= P_1 J^*J P_1= J^*J$. We observe that
$$\partial_1\partial_1^*=\Lambda^0_1(G)-A^0_1(G),$$
where $A^0_1(G)$ is the adjacency matrix and $\Lambda^0_1(G)$ is the degree (diagonal) matrix of $G$.

Therefore, if the graph $G$ is $k$-regular (i.e. if all its vertices have the same degree $k$) the adjacency matrix and its spectrum can be recovered from the $\langle P_0,P_1\rangle$-distribution of $\langle J,J^*\rangle$, since this last one includes the distribution of $\partial_1\partial_1^*$.

If the graph is not $k$-regular, one must either add the algebra generated by the degree matrix $\Lambda^0_1(G)$, or consider a more general notion of non-commutative distribution (e.g. in the sense of traffics \cite{CDM16}, where degree matrices appear quite naturally). 

In this sense, the usual study of the spectra of adjacency matrices of graphs in non-commutative probability does not fall far away from the study of the $\cB$-valued distribution of $\langle J,J^*\rangle$, which makes sense for higher dimensions and encodes important topological invariants. For example, the spectra of $\partial_1\partial_1^*:C_0\to C_0$ and $\partial_1^*\partial_1:C_1\to C_1$ have important interpretations.

Let $H=\{v_{i(1)},v_{i(2)},\dots,v_{i(k)}\}\subseteq G$ be a connected component of $G$. Since $\partial_1\partial_1^*=\Lambda^0_1(G)-A^0_1(G)$, it is easy to see that $\partial_1\partial_1^*(v_H)=0$, where the null-eigenvector $v_H=(\alpha_1,\alpha_2,\dots,\alpha_{n_0})^\dagger\in C_0$ is the indicator vector of $H$, with $\alpha_j=1,$ if $v_j\in H$ and zero otherwise. Hence, the number of connected components $\beta_0(G)$ of the graph $G$, can be read from the weight of the atom at zero of the spectral measure of $\partial_1\partial_1^*$.

The matrix $$\partial_1^*\partial_1=:\Lambda^1_0(G)-\mathrm{sgn}(A^1_0(G))=2I_{n_1}+\mathrm{sgn}(A^1_0(G)):C_1\to C_1$$ is the sum of the (constant, diagonal) degree matrix of the ($2$-regular) dual graph $G'$ of $G$ and an oriented (or signed) version of the adjacency matrix of $G'$, $A^0_1(G')=A^1_0(G)$ . 

The spectrum of $\partial_1^*\partial_1$ has also an important interpretation. The kernel of the boundary matrix $\partial_1$ is isomorphic to a subspace of (properly oriented) cycles  $Z_1\subseteq C_1$. Thus, the weight $\beta_1$ of the atom at zero of the spectral measure of $\partial_1^*\partial_1$ counts the number of holes $\beta_1(G)$ of the graph $G$. 

\subsection{General case ($d\geq 2$)} For a $d$-dimensional simplicial complex, consider its boundary matrix $J$. Observe that $JJ^*$ and $J^*J$ are the block-diagonal matrices
$$JJ^*=\left(
\begin{array}{cccc}
\partial_1\partial_1^* & 0 &\dots&0\\
0 & \ddots&0&0\\
\vdots & 0  &\partial_d\partial_d^*&0\\
0 & \dots &0&0
\end{array}\right),\quad J^*J=\left(
\begin{array}{cccc}
0&0 & \dots &0\\
0&\partial_1^*\partial_1 & 0&0\\
\vdots &0 & \ddots &0\\
0&0 & \dots &\partial_{d}^*\partial_{d}
\end{array}\right).$$ 

Once again, the weight $\beta_0(X)$ of the zero eigenvalue of the spectral measure of the positive operator $\partial_1\partial_1^*$ coincides with the number of connected components of $X$, which coincides with the number of connected components of its $1$-skeleton (the simplicial complex obtained from $X$ by erasing all its faces of dimension greater that $1$). 

In the same way, the dimension of the kernel of the positive operator $\partial_d^*\partial_d$ coincides with the number of $(d+1)$-dimensional holes in $X$. For example, for the void tetrahedron above, the spectral measure of $\partial_2^*\partial_2$ is $\frac{1}{4}(\delta_0+3\delta_4)$.

It turns out that there is a more general result and the operators $\partial_1\partial_1^*$, and $\partial_d^*\partial_d$ are just special corners of a block-diagonal operator $$\mathcal L:=JJ^*+J^*J=\left(
\begin{array}{ccccc}
\partial_1\partial_1^*&0 & \dots &0&0\\
0&\partial_1^*\partial_1+\partial_2\partial_2^* & 0&0&0\\
0&0 & \ddots &0&0\\
0&0 & \dots &\partial_{d-1}^*\partial_{d-1}+\partial_d\partial_d^*&0\\
0&0 & \dots &0&\partial_{d}^*\partial_{d}
\end{array}\right),$$
which encodes all the Betti numbers.

\begin{theorem}
Let $X$ be a finite simplicial complex and $J$ be its boundary matrix. If $\mu=(\mu_0,\dots,\mu_d)$ is the multivariate analytic distribution of the positive diagonal element $\mathcal L:=JJ^*+J^*J$, then the weight of the atom at zero of $\mu_i$ is $n_i^{-1}\beta_i(X)$, where $n_i$ is the number of $i$-dimensional faces of $X$ and $\beta_i(X)$ its $i$-th Betti number.
\end{theorem}

\begin{proof}
The theorem is a mere reformulation of results of Eckmann \cite{Eck44}, who computed the spectra of the higher-order combinatorial Laplacians (defined also in his work), which are exactly the compressed operators $\mathcal L_i:=P_i\mathcal L P_i=\partial_{i}\partial_{i}^*+\partial_{i-1}^*\partial_{i-1}:C_i\to C_i$. 
\end{proof}

In \cite{Fri98}, Friedman designed efficient algorithms for computing the spectra of combinatorial Laplacians, by the means of Hodge Theory. In \cite{HJ13}, the authors study the effects of weight functions and operations between simplicial complexes (e.g. wedges, joins and duplications of motifs) on the spectra of $\mathcal L_i$ or $J^*$.

Now we introduce the basic terminology from TDA.

\section{Basic models in TDA}\label{section:TDA}

We have shown elementary connections between algebraic topology and non-commutative probability theory. More sophisticated links can be found in \cite{Par15}. In particular, Boolean cumulants and the lattice of interval partitions $\IP$ appear in \cite{DPT15a}.

In the last decades, with the surge of topological data analysis, there has been an increasing interest on the statistical use of Betti numbers. In particular, this has boosted research on several concrete directions, such as: 
\begin{itemize}
\item Finding efficient algorithms for computing Betti numbers of large simplicial complexes.  
\item Adapting hypothesis tests and inference methods from statistics to treat persistence-diagram-valued random variables.
\end{itemize}
In order to be able to discuss, in the last section, possible directions in which non-commutative probability theory may help in TDA (or vice-versa), we provide toy examples for illustrating the very basic concepts of TDA and a first connection with non-commutative probability. 

Before doing so (and without the need of further notation), we show that histograms are random variables of Boolean nature, as they are images of conditional expectations w.r.t. to interval partitions on the real line.

\subsection{Histograms and interval partitions}

A histogram may be thought in the context of a rectangular probability space. Indeed, consider the algebra of Borel-measurable real-valued random variables $\cA_{\mathcal F}$ from Example \ref{ex:ovps} (1). Let $\cB:=\langle 1_{I(0)},1_{I(1)},\dots, 1_{I(k)} \rangle$ be the algebra generated by indicator functions to some pairwise disjoint bins (i.e. measurable sets) $\pi=\{I_0,I_1,\dots, I_k\}$ yielding a partition of the real line. 

The indicator functions are pairwise orthogonal projections and hence the conditional expectation $\EE_{\pi}:\cA_{\mathcal F}\to\cA_{\cB}$ maps a real random variable $X$ onto a random variable $\EE_{\pi}(X)=w_0\delta_{I(0)}+w_1\delta_{I(1)}+\dots +w_k\delta_{I(k)}$, with $k+1$ bins as possible outcomes. The probabilities of the bins are of course the expectations $(w_j)_{j\geq 1}=(\EE(1_{I(j)}))_{j\geq 1}$. 

In principle, bins may be quite arbitrary, even disconnected. However, it will be hard for us to construct a histogram from those bins in practice. For example, consider the partition with $n$ bins: $$I_i^{(n)}:=[\frac{i-1}{n^2},\frac{i}{n^2}]\cup[\frac{n+i-1}{n^2},\frac{n+i}{n^2}]\cup\dots\cup [\frac{n^2-n+i-1}{n^2},\frac{n^2-n+i}{n^2}],\quad i\leq n,$$ covering the unit interval and another bin $I_0=[0,1]^C$. The knowledge of the outcome of a random variable on the unit interval falling in any of these bins does not give much useful geometrical information about the outcome (in particular, the center of mass of  $I_i^{(n)}$ may be quite distant from it).

If we choose \emph{good} bins, the projection-valued random variable $\EE_{\pi}(X)$ can essentially be replaced by the real-valued discretization of $X$ $w_0\delta_{\lambda_0}+w_1\delta_{\lambda_1}+\dots +w_k\delta_{\lambda_k}$, where the $(\lambda_j)_{j\geq 1}$ are the centers of mass of the intervals $(I(j))_{j\geq 1}$.

In particular, a histogram $\EE_{\pi}(X)$ will be approximated by the empirical histograms $\EE_{\pi}(\bar X)=w_0'\delta_{\lambda_0}+w_1'\delta_{\lambda_1}+\dots +w_k'\delta_{\lambda_k}$, with weights proportional to the number of occurrences of each bin (asymptotically proportional, by the LLN, to the actual weights $(w_j)_{j\geq 0}$). 

If the number of independent experiments is very large, we may even make $k\to \infty$, a bit more slowly though, in such a way that we gradually refine the bins so that their sizes become arbitrarily small but the number of expected occurrences for each bin still grows arbitrarily large. Under such an asymptotic regime, the empirical process of histograms converges very, very slowly to the actual density of $X$ (if it exists). This can be thought as a general way to display a probability distribution, with obvious generalizations to higher dimensional euclidean spaces. A comparison with the much faster convergence of empirical eigenvalues of Wigner matrices is displayed in Fig. 1.

\begin{figure}\label{fig:convsuperconv}
\includegraphics[height=3cm]{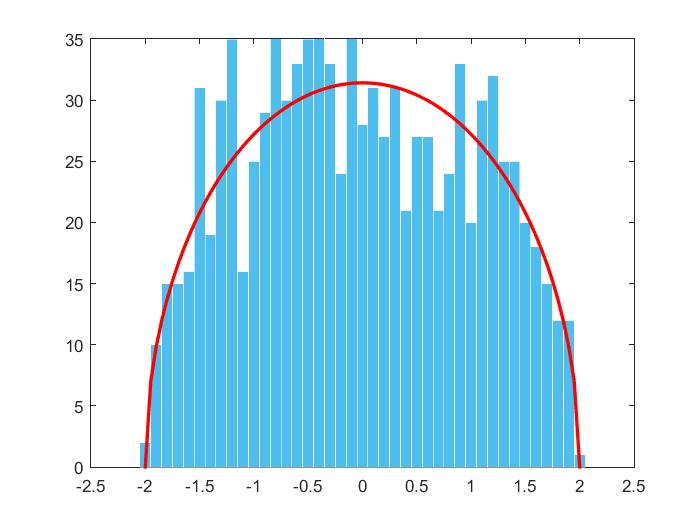}\includegraphics[height=3cm]{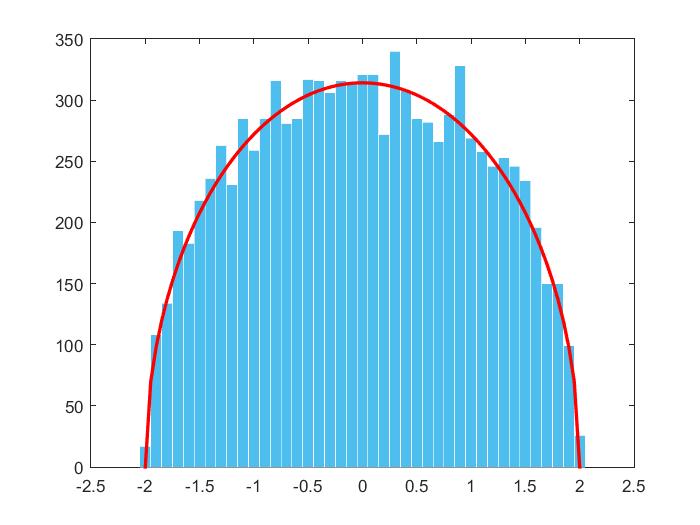}\includegraphics[height=3cm]{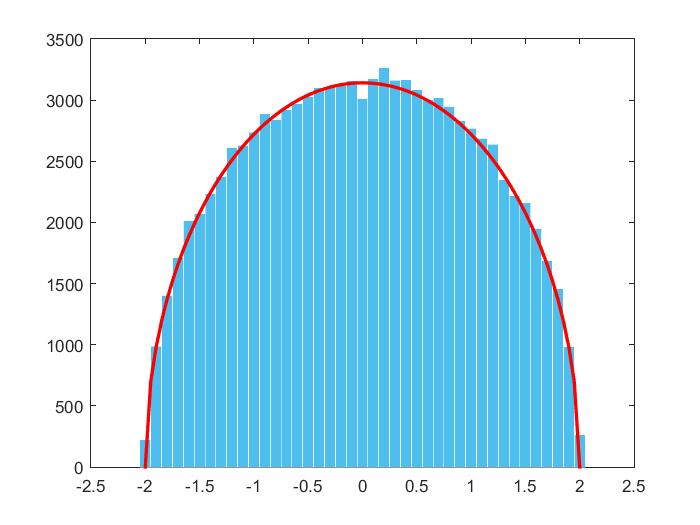}

\includegraphics[height=3cm]{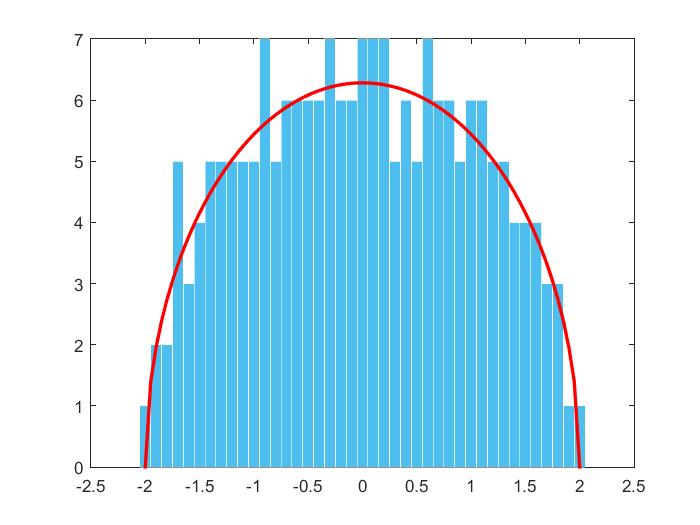}\includegraphics[height=3cm]{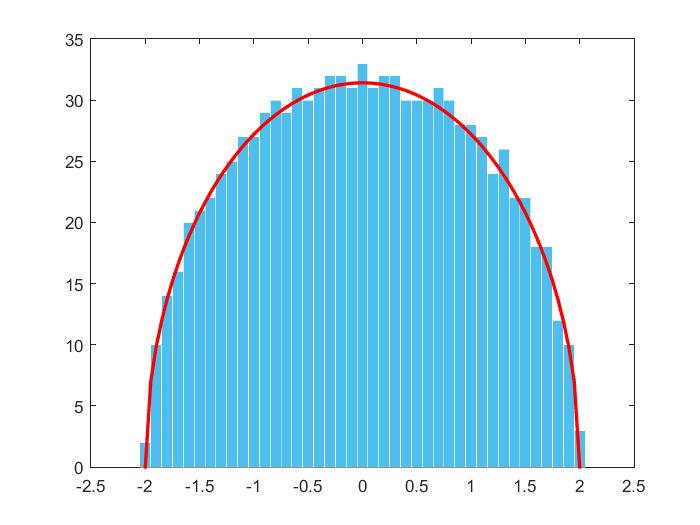}

\caption{Above: Uniform independent samples of semicircular random variables, $N=10^3,10^4,10^5$. Below: Repulsive eigenvalues of a single self-adjoint Gaussian matrix $W_{200}$ (left) $W_{1000}$ (right)}.
\end{figure}

When the outcomes of $X$ are now valued in a topological space, it may be not be clear in general how to retrieve \emph{effectively} the distribution which produced those outcomes. Hence, some simpler questions could be asked about the distribution. For example, a good estimate on the number of connected components of the support of $X$.

Similar to the asymptotic histogram process just described above, the notion of persistent homology allows us to recover, asymptotically, \emph{topological} properties of the support of $X$. We illustrate this with toy models in order to introduce the basic terminology on TDA, which will be needed in te last section. 

\subsection{A toy example in TDA}\label{subs:tensortoyTDA}

Suppose that $Y=(\bar X_i)_{i \leq N}$ are now independent outcomes of a random variable $X$ with (say, uniform) distribution on a given topological space (e.g. a torus in $\RR^3$).

The set $Y$ is called a point cloud. For each $t\geq 0$, we let  $B(\bar X_i,t)$, $i\leq N$ be closed balls of radius $t$, with centers at the outcomes $(\bar X_i)$. Then, for each $t\geq 0$ we consider the corresponding \v{C}ech complexes $(\Delta_{t})_{t\geq 0}$, which are simplicial complexes obtained in the following way:

The vertices of $\Delta_{t}$ are simply the points in $Y$. A $k$-dimensional simplex $(i_0,i_1,\dots, i_k)$ is included in $\Delta_{t}$ iff $\emptyset \neq \bigcap_{0\leq j\leq k}  B(\bar X_{i_j},t)$. 

Hence, as $t$ grows we will be adding more faces to the simplicial complex $\Delta_t$. For example: $\Delta_{0}$ consists only on the vertices and the first edge will be added when $2t\geq \min_{i,j} ||\bar X_i-\bar X_j||$.

For tame topological spaces and $N$ large enough, there will be some time interval $t_0 \leq t\leq t_1$, such that the support of $X$ and $Y_{t}$ (and hence $\Delta_{t}$) are homotopically equivalent. The endpoint $t_1$ is a constant related to the different diameters of the topological space in question (e.g. the inner and outer radii of a torus), whereas $t_0$ depends on the local distances between neighboring outcomes, and will converge to zero as $N\to \infty$.

Since $\Delta_{t}$ is a simplicial complex, we may calculate topological invariants with the computer, and visualize the evolution of the Betti numbers (or other invariants) as $t$ grows (see Figure 4)\footnote{The images of points on tori, their Betti curves and persistense diagrams are thanks to Yair Hernandez and Gilberto Flores.}.

The main idea of TDA is that one may use these topological summaries to construct statistics that allow us to classify point clouds. The way in which one actually does this is still an art-science, and the methods need to be adjusted quite carefully according to the nature of the data. 

Now we list some quite general facts on the implementation of operator-valued free probability to random matrices. These will also give us useful notation from random matrix theory to describe our toy models with repulsion in Section 7 and hopefully illustrate the algorithmic potential of non-commutative probability, which we would like to mimic for computational/stochastic algebraic topology.

\section{Large random matrices and free independence}\label{section:RMFP}

In the same spirit as the tensor product of matrices, Voiculescu noted in his seminal paper \cite{Voi91}, that free independence describes the asymptotic mixed normalized expected traces (moments) of matrices which have been randomly rotated by unitary matrices. Randomly rotated matrices can be thought as being placed in \emph{general position}.

Now let us summarize some classic results on the asymptotic global spectra of random matrices by reformulating these in terms of Voiculescu's free probability theory. For this, we need first the notion of convergence in non-commutative distribution.

\begin{definition}[Convergence in non-commutative distribution]
Let $(\cA_N,\tau_N)_{N\geq 1}$, $(\cA,\tau)$, $(a_1^{(N)},\dots,a_k^{(N)})\in \cA_N^{k}$, and $(a_1,\dots,a_k)\in \cA^k$  be collections of non-commutative random variables in different non-commutative probability spaces. We say that $(a_1^{(N)},\dots,a_k^{(N)})$ converges in distribution to $(a_1,\dots,a_k)$ iff all the mixed moments converge, that is, for all $m\geq1$ and $i:[m]\to[k]$, 
$$\lim_{N\to \infty}\tau_N(a_{i(1)}^{(N)},\dots,a_{i(m)}^{(N)})\to \tau(a_{i(1)},\dots,a_{i(m)}).$$
\end{definition}

The non-self-adjoint Gaussian matrix, also known as Ginibre matrix, $C_N:=\frac{1}{\sqrt N}(z_{ij})$ is the random matrix obtained from placing independent standard complex Normal random variables on the entries of an $N\times N$ matrix. In the literature, replacing the adjective ``Gaussian" by ``Wigner", means that we allow the $z_{ij}$ to be (not necessarily Normal) centered random variables with variance $\EE(z_{ij}z_{ij}^*)=1$. 

The non-self-adjoint Gaussian matrix is bi-unitarily invariant, that is, $U_NC_NV_N$ has the same distribution, for any choice of unitary matrices $U_N,V_N$ in the compact unitary group $\mathcal U_N$.

The random self-adjoint matrix $S_N=\frac{1}{\sqrt {2}}(C_N+C_N^*)$ is, essentially, the first random matrix for which an asymptotic spectral analysis was performed. Wigner showed in \cite{Wig58} that the averaged spectral distribution of $S_N$ converges to the standard (Wigner's) semicircle law.

Some decades later, Marchenko and Pastur studied in \cite{MP67} the spectra of modifications of the Wishart ensemble, by considering the models $C_ND_NC_N^*$, where $(D_N)\to d$ are self-adjoint deterministic matrices with some asymptotic analytic distribution $\mu_{D_N}\to \mu_d$. They observed that the asymptotic distributions depends on $D_N$ only through $\mu_d$ and computed the limiting law. 

The group $\mathcal U_N$ of unitary matrices is a compact topological group, and hence, a unique uniform (Haar) probability measure can be defined on it. It turns out that the (random) polar part of a non-self-adjoint Gaussian random matrix $U_N=C_N(C_N^*C_N)^{-1/2}$ has exactly the Haar distribution.

Now we are able to summarize Voiculescu's results in \cite{Voi91} in the following theorems and corollaries.

\begin{theorem}[Asymptotic freeness of random matrices and deterministic matrices \cite{Voi91}]\label{Th:Afree}
Let $C_1^{(N)},C_2^{(N)},\dots, C_{p+q}^{(N)}$ be independent non-self-adjoint Gaussian ensembles, considered as elements in the non-commutative probability spaces $(\cA_N,\tau_N):=(M_N(\CC)\otimes\mathcal A_{\mathcal F}, \frac{1}{N}\mathrm{Tr}\otimes \EE)$ of random matrices. Let us drop the $N$ super-index immediately, and let, for each $i\geq p$,
$U_i:=C_i(C_iC_i^*)^{-1/2}$ be the (individually Haar-distributed) polar parts of the first $p$ matrices. Then as $N\to\infty$ the convergence
$$(U_1,U_1^*\dots,U_p^*,C_{p+1},C_{p+1}^*,\dots,C_{p+q}^*)\to(u_1,u_1^*,\dots,u_p^*,c_{p+1},c_{p+1}^*,\dots,c_{p+q}^*)$$
holds in non-commutative distribution, where $u_1,\dots,u_p,c_{p+1},\dots,c_{p+q}\in\cA$ are \emph{free} Haar-unitaries and circular elements. 

Furthermore, if $D_1,\dots, D_{r}\in M_N(\CC)\subset\cA_N$ are deterministic matrices with some algebraic limiting distribution $(D_1,D_1^*,\dots, D_{r},D_r^*)\to (d_1,d_1^*\dots d_r, d_r^*)$ for some operators $d_1,\dots, d_r \in\cA$, then 
$$(U_1,\dots,U_p^*,C_{p+1},\dots,C_{p+q}^*,D_1,\dots,D_r^*)\to(u_1,\dots,u_p^*,c_{p+1},\dots,c_{p+q}^*,d_1,\dots,d_r^*),$$
where, again, $u_1,\dots,u_p,c_{p+1},\dots,c_{p+q}\in\cA$ are \emph{free} Haar-unitaries and circular elements, \emph{free} from the algebra $\langle d_1,d_1^*\dots d_r, d_r^* \rangle$.
\end{theorem}

\begin{corollary}
(1). The Wigner ensemble $S_N=\frac{1}{\sqrt{2}}(C_N+C_N^*)$ converges in distribution to $s=\frac{1}{\sqrt{2}}(c+c^*)$, with semicircular distribution, which plays the role of the Gaussian distribution in free probability.

(2). The algebras $\cA_1:=\langle C_N, C_N^*\rangle$ and $\cA_2:=\langle D_N \rangle$ are asymptotically free (that, is, they converge to free sub-algebras). This leads to nice formulas for the asymptotic distribution of $C_ND_NC_N^*$. In particular, Speicher's combinatorial machinery leads to elementary computations of the asymptotic moments, which explains the possibility to formulate the functional equations for the Cauchy-Stieltjes transform derived by Marchenko and Pastur. In the free probability framework, the asymptotic laws of $C_ND_NC_N^*$ correspond to the free analogues of the compound Poisson distributions.

(3). Elementary computations show that if $\langle u,u^* \rangle$ and $\langle a,b \rangle$, $a=a^*,b=b^*$ are free, then $\langle a \rangle$ and $\langle ubu^* \rangle$ are free. Hence, in view of Theorem \ref{Th:Afree}, we may produce realizations of free non-commutative random variables, by randomly conjugating deterministic matrices $D_1,D_2$ with (say, diagonal) entries, converging to some prescribed distributions $\mu_1,\mu_2$. 

Therefore, the free additive and multiplicative convolutions $\mu_1\boxplus\mu_2$ and $\mu_1\boxtimes\mu_2$, can be \emph{realized} as the asymptotic eigenvalue distributions of $D_1+U_ND_2U_N^*$ and $D_1^{1/2}U_ND_2U_N^*D_1^{1/2}$. 

\end{corollary}

The celebrated ``circular law'' (see \cite{Gin65}, or \cite{TV10}, for minimal assumptions) indicates that the distribution of a random eigenvalue from the random matrices $C_N$ converges as $N\to \infty$ to the uniform distribution on the unit disk. Non-normal matrices are spectrally unstable and hence, the derivation of the circular law follows quite different techniques in random matrix theory. 

Furthermore, the non-self-adjoint Gaussian matrix $C_N$ provides not only a way to produce the uniform distribution on the unit disk, by selecting a random eigenvalue, but a possibility of producing a cloud of $N$ repelling points, covering the unit disk uniformly, by selecting all of them (see Figure 2).

The repulsion phenomena implies that stronger notions of convergence happen in the random matrix framework (see Figure 1). For simplicity, we will refer to this phenomenon in general as ``superconvergence" (although, strictly speaking, this term was used by Bercovici and Voiculescu for describing some analytic regularization and spectral stability of operators produced by the free convolution).

\begin{figure}\label{fig:circlaw}
\includegraphics[height=4.5cm]{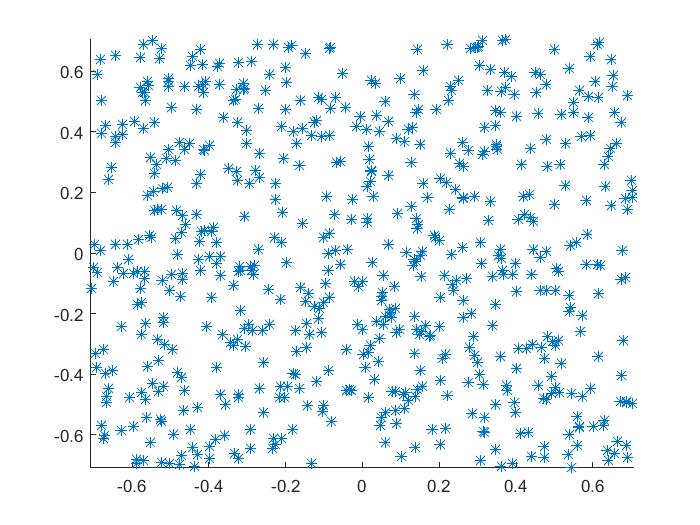}\includegraphics[height=4.5cm]{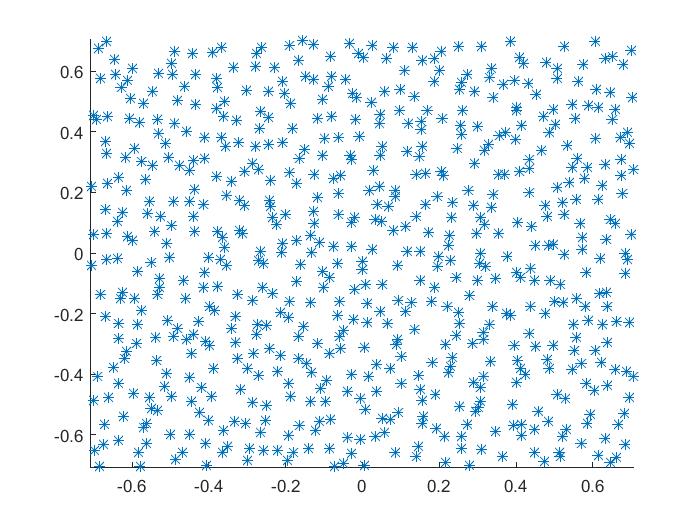}
\caption{Uniform independent samples (left) vs uniform repulsive eigenvalues of $C_{1000}$ (right)}
\end{figure}

This will be later combined with methods for transferring uniform distribution from the unit disk to a parametrized surface. For example, for the acceptance-rejection method \cite{DHS13}, the repulsion of the points on the disk will be inherited after the points are mapped to the surface.
We will use these matrix models to construct toy models with repulsion in TDA (see Section \ref{subs:repulsionTDA}). For dimension $d=1$, we may use the same idea with the eigenvalues of the polar part $U_N$, which are distributed uniformly with repulsion on the unit circle.

\subsection{Asymptotic $\cB$-freeness and the free probability pipeline}\label{subsection:OVFP}

Voiculescu's theorem describes the asymptotic traces of products of matrices in general position. However, the assumption of general position will be soon broken as we depart from the classical models of Wigner and Marchenko-Pastur to more sophisticated models.

For example, we may simply combine some independent $N\times N$ non-self-adjoint matrices $C_1,C_2,C_3,C_4,C_5,C_6$ and form block $3N\times 3N$ random matrices $$C=\left(
\begin{array}{ccc}
C_1&C_2 & C_3\\
C_2&C_1 & C_1\\
C_3&C_1 & C_2
\end{array}\right),\quad C'=\left(
\begin{array}{ccc}
C_4&C_4 & C_5\\
C_4&C_5 & C_6\\
C_5&C_6 & C_6
\end{array}\right).$$

Since some of the entries of $C$ and $C'$ are no longer independent, the limiting distribution of $C+C^*$ is no longer a semicircle law and the matrices $C,C'$ are no longer asymptotically free. However, by Thm. \ref{Th:Afree} we know that the blocks $C_1,C_2,C_3,C_4,C_5,C_6$ converge in distribution to free circular elements $c_1,c_2,\dots,c_6$.

In view of Remark \ref{rem:indeps} (3), for $A_1=\langle c_1,c_2,c_3,c_1^*, c_2^*, c_3^* \rangle$ and $A_2=\langle c_1,c_2,c_3,c_1^*, c_2^*, c_3^* \rangle$, Shlyakthenko noticed in \cite{Shl96} that $C, C'$ are free over $\cB=M_3(\CC)$, and possibly over a smaller sub-algebra, depending on the patterns in which we insert the blocks, and he derived the fixed point equations to compute the distributions.

This illustrates quite well a general pipeline for computing distributions in asymptotic random matrix theory. A random matrix model will usually be constructed by several pieces, as above (block-matrix) or by evaluating some non-commutative polynomial on these matrices. For example, the matrices $C_N$ and $D_N$ are the pieces of the Marchenko-Pastur model and these pieces are evaluated on the polynomial $P(x,y,y^*)=yxy^*$. 

Then the model $a$ (or a suitable modification $a'$) is shown to be free over a certain algebra and the cumulant computations in order to describe the distribution of the model take place at that operator-valued level. In particular, one usually computes the $\cB$-valued Cauchy transform $G_a^{\cB}(b):\FF((b-a)^{-1})$ and the desired scalar-valued transform is a linear algebraic function of it (for example, the normalized trace $G_a(z)=\frac{1}{N}\mathrm{Tr}\circ G_a^{\cB}(z)$ or the upper corner $G_a(z)=(G_{a'}^{\cB}(z))_{11}$)

Probably the most impressive practical application of free probability to random matrices is the algorithm designed by Belinschi et al. \cite{BMS15} to tackle arbitrary polynomials on free random variables. 

Hence, operator-valued freeness extends the applicability of the theory quite a lot. The algorithmic problems in TDA and free probability have some similarities. They are both related to the problem of finding statistics about diagonalizations of large matrices (the Smith normal form on one side and the Jordan normal form or the singular value decomposition on the other side). 
However, even with the flexibility gained with operator-valued free probability, the dependence on the entries of the random boundary/incidence matrices appearing in TDA are still of a completely different nature.

Nevertheless, we should point out similarities with Patels result, which is kind of operator-valued in the sense described above, that the M\"obius inversions take place at multivariate levels of  homology groups. This analogy should later be complemented with an appropriate translation of the methods in computational algebraic topology in terms of non-commutative probability.

In the last section, we point out some interesting problems in non-commutative probability which could give some more understanding of the persistence phenomena and its applications.

\section{Future work}

In this last section, we discuss some problems which are relevant for developing new connections between non-commutative probability and topological data analysis.

\subsection{Superconvergence for toy models with repulsion}\label{subs:repulsionTDA}

For a nice topological space $X$ of dimension $d=1,2$, in the sense that we are able to describe it by a regular parametrization, it is possible to produce point clouds such that the distribution of a random point of the cloud is uniform on $X$, but such that the different points in the cloud repel each other when considered simultaneously. For loops or intervals, one may simply cut the unit circle with the eigenvalues of a Haar-distributed unitary matrix $U_N=C_N(C_NC_N^*)^{-1/2}$ and paste it along a 1-dimensional face in a topological space (edge or loop).

For dimension $d=2$ we may consider the repulsive eigenvalues on the unit disk, generated by the non-self-adjoint Gaussian matrix model $C_N$. We may input such points on the acceptance rejection method see \cite{DHS13}, which yields an uniform distribution on a parametrized surface (see Figure 3).

\begin{figure}\label{fig:750torus}
\includegraphics[height=4cm]{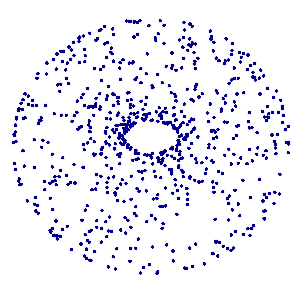}\includegraphics[height=4cm]{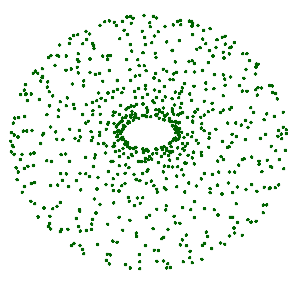}

\caption{Mapping 750 points with uniform distribution on the torus. Independent (left) and with repulsion (right).}
\end{figure}

The repulsion on the input data gets transferred to the accepted data, mapped on the surface. The repulsion imbues the point clouds and hence the associated filtrations of complexes with additional regularity, which makes the convergence of the Betti curves much faster than in the ``tensor" toy model (specially for homologies of dimension $1$ and $2$, which are more sensitive to regularity, see Figures 4 and 5).

\begin{figure}\label{fig:Bettibarcodes}
\includegraphics[height=4cm]{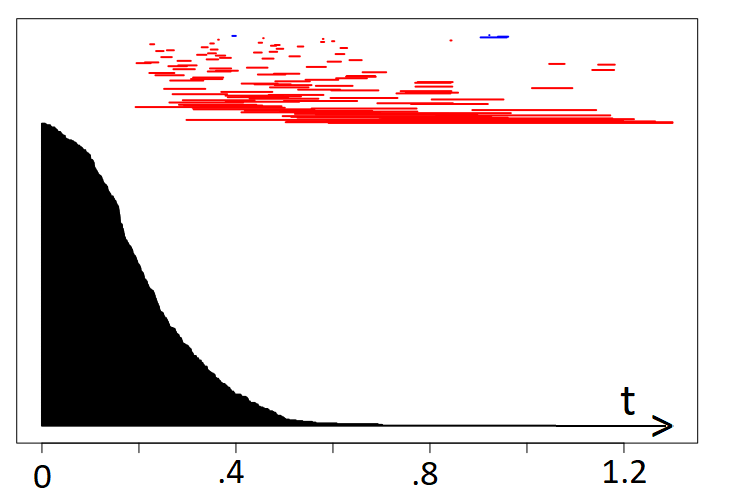}\includegraphics[height=4cm]{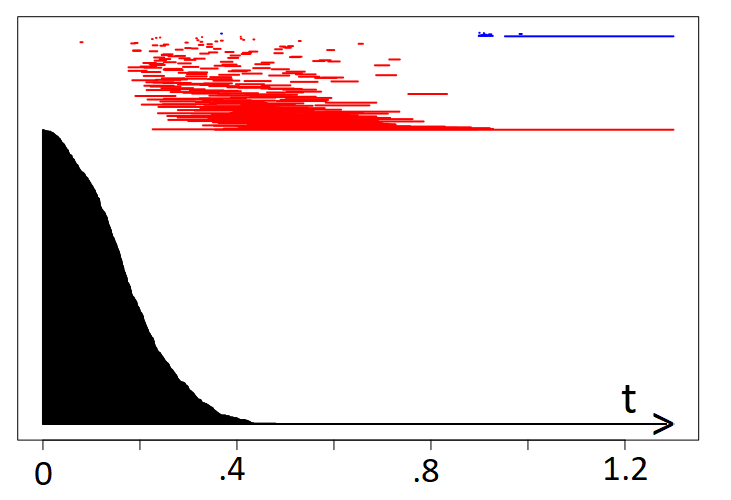}

\includegraphics[height=4cm]{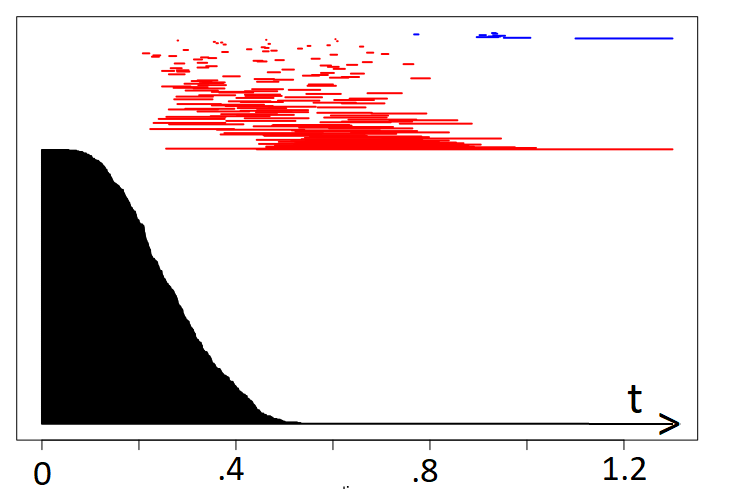}\includegraphics[height=4cm]{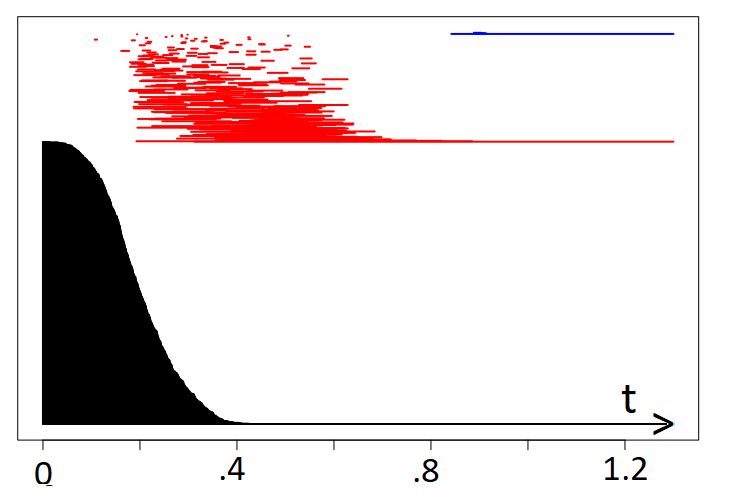}

\caption{Comparison between Betti barcodes for 400/750 (left/right) points on the torus for $H_0$ (black), $H_1$ (red) and $H_2$  (blue). For independent points (above) and points with repulsion (below).}
\end{figure}

Geometric models with repulsion have been considered already in stochastic topology. Very often, data is given by point clouds with coordinates intentionally arranged into evenly spaced grids (as for the Mexican masks). The geometrical repulsion is usually produced by considering the tensor model and simply dismissing points which have fallen too close to a previous point\footnote{I learned about this during a mini-course of M. Kahle at CIMAT}. 

\begin{figure}\label{fig:persistencediag}

\includegraphics[height=4cm]{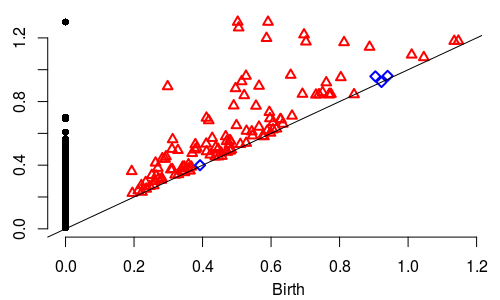}\includegraphics[height=4cm]{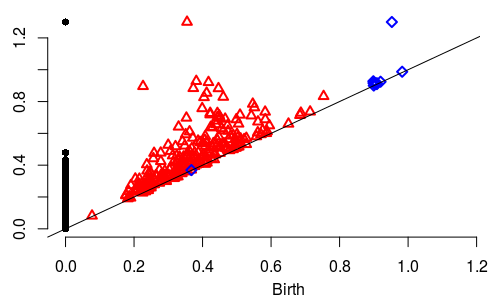}

\includegraphics[height=4cm]{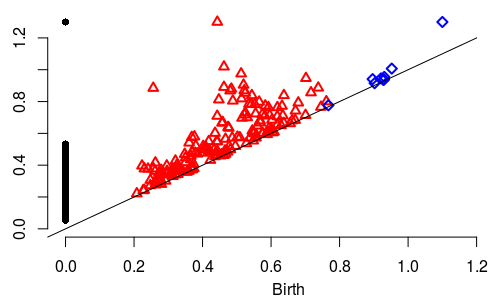}\includegraphics[height=4cm]{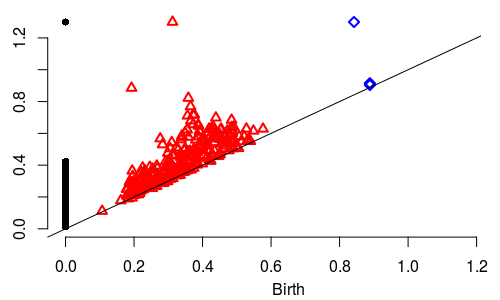}

\caption{Persistence diagrams for 400/750 (left/right) points on the torus for $H_0$ (black), $H_1$ (red) and $H_2$  (blue). For independent points (above) and points with repulsion (below).}
\end{figure}

Hopefully, some of the rich symmetries enjoyed by models produced in terms of eigenvalues of random matrices will lead to better estimates for the asymptotic regimes relevant in TDA. On the downside, it is not clear how to produce such models for $d\geq 3$.

\subsection{Changing the Field}

In Section 3, we extracted the Betti numbers of a simplicial complex by computing the spectra of the blocks of the combinatorial Laplace operator $\mathcal L:=JJ^*+J^*J$.

However, the Betti numbers can be obtained directly from $J$, by computing the Smith normal forms of all the boundary operators $\partial_i$. The Smith normal form of an $n\times m$ matrix $A$ with entries in a principal ideal domain $R$, is a decomposition $A=SDT$, where $S\in M_n(R)$ and $T\in M_m(R)$ are invertible matrices and $D$ is an $n\times m$ matrix of the form 
$$\left(
\begin{array}{cc}
\tilde D&0 \\
0&0 
\end{array}\right),$$
where $\tilde D=\mathrm{diag}(d_1,d_2,\dots,d_k)$ has non-zero diagonal entries, such that $d_1|d_2|\dots |d_k$. Canonical choices for the PID are field of two elements $\FF_2$ and $\mathbb Z$ (or $\CC$).

Since $S$ and $T$ are invertible, the rank $k$ of $D$ coincides with the rank of $A$. For boundary matrices $\partial_i=S_iD_iT_i$, the rank $k_i$ does not depend on the choice of the principal ideal domain $R\in \{\FF_2,\mathbb Z\}$. In fact, we have that, for all $i\leq d$ $$\beta_i(X)=n_i-k_i-k_{i-1}$$

For $R=\mathbb Z$ the coefficients $d_1^{(i)},d_2^{(i)},\dots,d_{k_i}^{(i)}$ of the Smith normal form of $\partial_i$ are the torsion coefficients of the $i$-th homology group. If we use $R=\mathbb F_2$, the coefficients are all ones and hence this information is lost. However, by restricting to coefficients in $\FF_2$ one gains access to methods for sparse matrices in order to compute Smith normal forms more effectively. 

The fact that we may change the field is quite unprecedented in non-commutative probability theory (where we are used to work with $\CC$), even though the theory rests quite fundamentally on multiplicative functions on incidence algebras, which have in other instances quite prominent arithmetic consequences (when moment generating series are replaced by L-series and set partitions by the poset of natural numbers with division). 

No concrete applications of non-commutative probability to L-series have been found yet. However, we should point out the following remarks:

(1). The fundamental arithmetic operation $(a,b):=\max \{d: d|a,b\}=:ab/[a,b]$ is captured by elementary means in non-commutative probability theory. A $k$-Haar unitary (i.e. a random variable with uniform probabilities on the $k$-th roots of the unit) is a standard Gaussian random variables in the Boolean sense, since all its Boolean cumulant vanish, except $B_k(u_k,\dots,u_k)=1$. The product of \emph{classically independent} $k$-Haar unitaries follows the arithmetic rule $u_ku_l=u_{[k,l]}$. In particular, independent $k$ and $l$-Haar unitaries are multiplicative if $k$ and $l$ are co-prime.

(2). Important numbers in arithmetic, such as the Bernoulli numbers, begin to appear in the cumulant-to-cumulant formulae (see \cite{AHLV15}) when we consider non-symmetric notions of independence, such as the monotone independence (see \cite{Mur01, HS11}).

(3). Interval and non-crossing partitions have nice arithmetic features. In particular, the sub-posets of $k$-divisible partitions may be considered and their M\"obius theory can be calculated (see \cite{Ari12}). These have been used in order to compute moments of products of free random variables \cite{AV12}.

We expect more important applications of these posets and their multiplicative cumulant theory, possibly by relaxing/modifying some axioms on the algebraic products of \cite{BGS02}, to be able to consider notions of independence more directly in terms of $k$-nary products (and not necessarily iterating binary associative products).

(4). And, of course, at a much higher level of complexity, there is the very challenging problem of understanding the remarkable statistical coincidences, first observed by Montgomery, between the eigenvalues of Haar-distributed unitary matrices $U_N=C_N(C_NC_N^*)^{-1/2}$, which are central in free probability, and the Riemann zeta function (see \cite{Dia03} for a survey).

\subsection{Star products of graphs and Boolean cumulants, revisited}

As we saw from the star products of graphs, Boolean cumulants share properties with the Betti numbers. In particular, they are additive with respect to ``star" products of simplicial complexes (i.e. identifications of two simplicial complexes by a selected simplex). 

Boolean cumulants filter cycles in a graph according to a local topological property. In view of Section 4, we conjecture that the Betti numbers can be obtained as cumulats with respect to a suitable conditional expectation (for example, onto an algebra determined in some way by the incidence/boundary matrix).

A first step would be to compare the Boolean cumulants and the first Betti number. A second step, to define Boolean cumulants for higher dimensions, using appropriate partitions (\cite{GMV17}). 

The potential advantage of factoring Betti numbers into Boolean cumulants is that they actually sort cycles by combinatorial length (i.e. the number of edges in the cycle). Large cycles are usually persistent and hence, persistence homology would quite transparently correspond to simply ignoring the first few cumulants.

More generally, it would be interesting to filter cycles according to more topological properties in terms of cumulants. For example, since we already know how to filter cycles which are not simple at the vertex $v_i$ by considering the Boolean cumulants w.r.t. $\tau_{ii}$, a cleaver combination of the $\tau_{ii}$'s (possibly in the algebraic framework of traffics, were cumulants somehow model logic gates) should filter simple curves.

\subsection{Betti forests instead of Betti curves}

The different-notions of non-commutative independence are combinatorially interrelated in different ways, but quite universally described by means of forests and trees related to partitions (see \cite{AHLV15}). Forests may display topological data in a better way than curves. 

By using forests, we may obtain quite some more geometrical information about the different clusters in a point cloud which would otherwise make them hard to distinguish. Think of five big clusters in a data cloud, such that, eventually, only points close the one big cluster are added.  Just by looking at the $0$-th Betti curve, it will be hard to distinguish this situation from the situation where more points are added independently to all the five clusters. 

Forest-codes of Betti numbers would allow us also to differentiate the situations:
\begin{itemize}
\item There is a loop at $t_0$ and then at $t_1>t_0$ it splits into two as we add an intermediate edge.
\item There is a loop at $t_0$ and then at $t_1>t_0$ a second loop appears anywhere else in the space.
\end{itemize}
This can not be done by simply looking at the Betti curve of dimension $1$. 

Analogously (dualy), forest-codes could distinguish the situations: 
\begin{itemize}
\item Two non-homologous loops at $t_0$ are made homologous at $t_1$ by adding the $2$-dimensional faces that separated them.
\item There are two non-homologous loops at $t_0$ and one of them gets filled up at $t_1$ by $2$-dimensional faces and becomes trivial.
\end{itemize}

There are currently no methods for computing such forest codes with the impressive efficiencies that are achieved for Betti curves. Theoretically, however, there is some recent work in this direction, by using chiral merge trees \cite{Cur17} (in the same framework as Patel's inversion formula).

\end{document}